\documentclass[11pt]{amsart}
\usepackage{amsfonts}
\usepackage{fullpage}
\usepackage{bbm}
\usepackage{mathrsfs}
\usepackage{color}
\usepackage{comment}

\usepackage{array}

\usepackage{xypic,amsfonts,amsmath, amssymb}

\newcommand{\bea}{\begin{eqnarray}}

\newcommand{\eea}{\end{eqnarray}}

\newcommand{\be}{\begin {equation}}

\newcommand{\ee}{\end{equation}}

\newcommand{\N}{\Bbb N}
\newcommand{\C}{\Bbb C}

\newcommand{\Z}{\mathbb{Z}}

\newtheorem{theorem}{Theorem}[section]

\newtheorem{corollary}[theorem]{Corollary}
\newtheorem{lemma}[theorem]{Lemma}
\newtheorem{proposition}[theorem]{Proposition}
\newtheorem{remark}[theorem]{Remark}

\newtheorem{example}[theorem]{Example}

\newtheorem{conjecture}[theorem]{Conjecture}

\newcommand{\am}{\mathcal{W}(p)^{A_m}}
\newcommand{\triplet}{\mathcal{W}(p)}
\begin{document}

\title[$ADE$ subalgebras of the triplet vertex algebra $\triplet$: $D$-series]{$ADE$ subalgebras of the triplet vertex algebra $\triplet$: $D_m$-series}
\author{Dra\v zen Adamovi\' c}
\address{Department of Mathematics, University of Zagreb, Bijenicka 30, 10000 Zagreb, Croatia}
\email{adamovic@math.hr}
\author {Xianzu Lin }
\address{College of Mathematics and Computer Science, Fujian Normal University, Fuzhou, {\rm 350108}, China}
\email{linxianzu@126.com}
\author{Antun Milas}
\address{Department of Mathematics and Statistics,SUNY-Albany, 1400 Washington Avenue, Albany 12222,USA}
\email{amilas@albany.edu}
\date{ }
\maketitle

\begin{abstract}

We are continuing our study of ADE-orbifold subalgebras of the triplet vertex algebra $\triplet$.
This part deals with the dihedral series.
First, subject to a certain constant term identity,  we classify all irreducible modules for the vertex algebra $\overline{M(1)} ^+$, the $\Z_2$--orbifold of the singlet vertex algebra $\overline{M(1)}$. Then we classify  irreducible modules and determine Zhu's and $C_2$--algebra for the vertex algebra $\triplet ^{D_2}$.  A general method for construction of twisted $\triplet$--modules is also introduced. We also discuss classification of twisted $\overline{M(1)}$--modules including the twisted Zhu's algebra $A_{\Psi} (\overline{M(1)})$ , which is of independent interest. The category of admissible $\Psi$-twisted $\overline{M(1)}$-modules is expected to be semisimple.  We also prove $C_2$-cofiniteness of $\triplet^{D_m}$ for all $m$, and   give a conjectural list of irreducible $\triplet^{D_m}$-modules. Finally, we compute characters of the relevant irreducible modules
and describe their modular closure.

\end{abstract}

\tableofcontents

\section{Introduction}

The family of triplet vertex algebras $\mathcal{W}(p)$, $p \geq 2$,  has recently attracted a lot of interest in connection
to logarithmic conformal field theory, Hopf algebras and tensor categories (see \cite{am1}, \cite{am2}, \cite{CR}, \cite{FGST1}, \cite{FGST2}, \cite{GK1}, \cite{GK2}, \cite{FSS}, \cite{Fu}, \cite{GKM}, \cite{NT}).

In \cite{ALM}, we started to investigate subalgebras of the triplet vertex operator  algebra $\triplet$ obtained as the fixed point subalgebra $\triplet ^{\Gamma}$, where $\Gamma \subset \mbox{Aut}(\triplet) \cong PSL(2, \C)$ is finite (so it is of ADE type). In the same paper, we proved the $C_2$-cofiniteness of $\triplet ^{A_m}$ (cyclic case) and presented several results and conjectures about its representation theory. In this installation we
prove analogous results for the subalgebras $\triplet ^{D_m}$, where $D_m$ is a dihedral  group of order $2m $.

Although $\triplet ^{D_m}$ might appear to be more complicated compared to $A$-series, it is important to point out that $\triplet ^{D_m}$ is in fact a $\Z_2$-orbifold of $\triplet ^{A_m}$ for the automorphism $\Psi$ of order two constructed in \cite{ALM}. If we adopthis point of view, investigation of $\triplet ^{D_m}$ has many similarities with the vertex algebra $V_L  ^{+}$ obtained as a $\Z_2$-orbifold of the rank one lattice vertex algebra $V_L $ (cf. \cite{dn2}, \cite{dg}). The vertex algebra $V_L ^+$ is known to be rational and $C_2$--cofinite (loc.cit).
Thus,
\begin{center}
{\em $\triplet ^{D_m}$ can be viewed as an irrational (or logarithmic) generalization of  $V_L ^+$}.
\end{center}
In order to analyze  the orbifold $V_L^+$, C. Dong and K. Nagatomo in  \cite{dn}, \cite{dn2} made a heavy use of representation theory of the vertex algebra $M(1) ^+$ , a $ \Z_2$--orbifold of the $c=1$ Heisenberg (bosonic) vertex algebra $M(1)$. So, in parallel with their approach, we first study the vertex operator algebra $\overline{M(1)} ^+$ (of central charge $c_{p,1}$!), a $\Z_2$--orbifold of the singlet vertex operator algebra $\overline{M(1)}$ studied extensively in \cite{a} (see also \cite{am0}).
 Recall that irreducible $\overline{M(1)}$--modules are of lowest weight type with respect to $(L(0), H(0))$
 $$L_{\overline{M(1)} }(x,y), \ \ y ^2 = C_p P(x),$$
 where $C_p$ is a non-zero constant and $P(x)$ is a polynomial of degree $2p-1$ (loc.cit.).
 In Section \ref{m(1)}, among other things, we proved that $\overline{M(1)} ^+$  is the only proper vertex subalgebra of $\overline{M(1)}$ which contains Virasoro vertex operator algebra $L(c_{p,1},0)$.  The Zhu algebra $A(\overline{M(1)} ^ + )$ is also commutative and generated by two vectors $[\omega]$, $[H ^{(2)}]$. But the explicit determination of relations among them seems to be a difficult problem. To tame this, we use approach similar to that developed in \cite{a}, \cite{am1}, \cite{am2} and in other papers.  Then up to a certain constant term identity (Conjecture \ref{slutnja-1}), we present a classification result for irreducible $\overline{M(1)}^ +$--modules.
Since Conjecture \ref{slutnja-1} can be easily checked on computer for small values of $p$, we get a large class of vertex algebras having similar representation theory as $M(1) ^+$ for $c=1$.

In this approach,  we also need twisted $\overline{M(1)} $--modules.
For every $1 \le j  \le p$ we construct   unique $\Psi$--twisted $\overline{M(1)}$--module $L_{\overline{M(1)} } ^{tw} (h_{p + 1/2-j ,1})$ of lowest weight $ h_{p +1/2-j,1}$.

 \begin{theorem}
 Assume that $p$ is small (i.e., $p \le 20$). Then
 \item[(1)] Twisted Zhu's algebra $A_{\Psi} (\overline{M(1)} )$  is semi-simple and isomorphic to
 $${\C}[x] / \langle h_p(x) \rangle $$
 where $h_p(x) = \prod_{i=1} ^p (x-h_{p+1/2-i,1})$.

 \item[(2)] The set  $\{ L_{\overline{M(1)} } ^{tw} (h_{p + 1/2-j ,1}) \ \vert \ j=1, \dots, p  \}$ provides all $\tfrac{1}{2}\N$--graded irreducible $\Psi$--twisted $\overline{M(1)}$--modules.

 \end{theorem}
 We shall prove this theorem for $p$ small, because then we can verify some technical conditions on computer. Of course, we conjecture that this theorem holds for every $p$.

 The above twisted $\overline{M(1)}$--modules, as an $\overline{M(1)}^+$--module,   decompose into a direct sum of two non-isomorphic irreducible modules $$L_{\overline{M(1)} } ^{tw} (h_{p + 1/2-j ,1}) = R(j)^{\sigma}  \oplus R(2 p+1 - j) ^{\sigma}. $$

\begin{theorem}
Assume that Conjecture \ref{slutnja-1} holds (verified for $p$ small).
 \item[(1)] The set
 $$ \{  L_{\overline{M(1)} } (h_{i,1}, 0) ^{\pm}, \ L_{\overline{M(1)} }  (x, y) , \ R(j) ^{\sigma},  y \ne 0, i= 1, \dots, p, j=1, \dots, 2p   \} $$
 gives a complete list of non-isomorphic irreducible $\N$--graded $\overline{M(1)} ^+$--modules.

  \item[(2)] Every irreducible $\overline{M(1)} ^+$--module is realized as a submodule of either irreducible $\overline{M(1)}$--module or irreducible $\Psi$--twisted $\overline{M(1)}$--module.
\end{theorem}

Next we study the representation theory of $\triplet ^{D_m}$. The main results in this directions are
proven in Sections  \ref{section-d2}, \ref{general-1} and \ref{general-2}. We prove $C_2$--cofiniteness for all $m$ and construct a family of irreducible modules.   In the case $m=2$, we present results on classification of all irreducible modules.

 \begin{theorem}
Assume that Conjecture \ref{slutnja-1} holds (verified for $p$ small).

\item[(1)] The vertex operator algebra $\triplet ^{D_2}$ has $11p$ irreducible representations.

\item[(2)] Zhu's algebra $A(\triplet ^{D_2})$ is commutative and
$$ \dim A(\triplet ^{D_2}) = \dim \mathcal{P} (\triplet ^{D_2}) = 12 p-1. $$

\item[(3)] The modular closure of irreducible characters of $\triplet ^{D_2}$-modules is $6 p-1$--dimensional.
\end{theorem}

   So up to a combinatorial  identity,  $\triplet ^{D_2} $, $p \in {\N}$,  gives a  new series of non-rational $C_2$-vertex operator algebras, for which we have complete classification of irreducible representations and description of Zhu's and the $C_2$-algebra. We plan to extend this result for general $m$ in our forthcoming papers.

\section{The vertex algebras $\overline{M(1)} ^+$ and  $\triplet ^{D_m}$}
\label{m(1)}

In this part we mostly follow \cite{am1}, \cite{am2}. We fix the rank one even integral lattice $L=\mathbb{Z} \alpha$ such that $\langle \alpha,\alpha \rangle=2p$.
Extension of scalars on $L$ leads to the abelian Lie algebra $\frak{h}$ and its affinization $\widehat{\frak{h}}$. Denote by $M(1)=S(\frak{h}_{<0}) \cong \mathbb{C}[\alpha(-1),\alpha(-2), \cdots]$
the Fock space, with the usual vertex algebra structure. The conformal structure is determined by
$$\omega=\frac{1}{4p} \alpha(-1)^2{\bf 1}+\frac{p-1}{2p} \alpha(-2){\bf 1}$$

Conformal vertex algebra $M(1)$ embeds into the lattice vertex algebra $V_L=M(1) \otimes \mathbb{C}[L]$, where $L$ denotes
the group algebra generated by $e^{\alpha}$, $\alpha \in L$. We use standard notation $Y(e^{\alpha},x)=\sum_{n \in \mathbb{Z}} e^{\alpha}_n x^{-n-1}$.
We have two (degree one) screenings: $Q=e^{\alpha}_0$ and $\tilde{Q}=e^{-\alpha/p}_0$ (strictly speaking $e^{-\alpha/p}_0$ operates on the generalized
vertex algebra associated to the dual lattice of $L$). The kernel of $\tilde{Q}$ on $V_L$ is what we call the triplet
vertex algebra and is denoted by $\mathcal{W}(p)$ \cite{FGST1}, \cite{FGST2}. For more about the structure of $\triplet$ see \cite{am1} \cite{FGST1}, \cite{FGST2} for example.
Recall that the triplet vertex algebra $\triplet$ has $2p$ irreducible modules: $$\Lambda(i), \Pi(i),  \quad i=1, \dots, p. $$

In \cite{ALM}, we have constructed an automorphism $\Psi$ of $\triplet$ such that
  $$ \Psi (Q ^i  e ^{-n \alpha} ) = \frac{(-1) ^i i!}{(2n-i)!} Q ^{2n-i} e ^{-n \alpha}, $$
and proved that  $\mbox{Aut}(\triplet) = PSL(2, \C)$.

In what follows for an untwisted $V$-module $M$ and $\sigma \in Aut(V)$, we denote by $M^\sigma$  the $V$-module where the actions is twisted
by  $\sigma$. One can easily show that $$\Lambda(i) ^{\Psi} \cong \Lambda(i), \ \ \Pi(i) ^{\Psi} \cong \Pi(i).$$ This implies that $\Psi$ induces $\triplet$--isomorphisms
of $\Psi_{\Lambda(i)}, \Psi_{\Pi(i)}$ of irreducible $\triplet$--modules. In particular, on $\Lambda (i)$ we have
 $$ \Psi_{\Lambda(i)} ( Q ^j e ^{\tfrac{i-1}{2p} \alpha - n \alpha} ) = \frac{(-1) ^j j !}{(2n -j)!} Q ^{2n -j} e ^{\tfrac{i-1}{2p} \alpha - n \alpha}. $$

The operators $Q$ and $f = -\Psi_{\Lambda(i)} \circ Q \circ \Psi_{\Lambda(i)}$ induce an $sl(2)$--action on $\Lambda(i)$. Similarly, we have an $sl(2)$--action on $\Pi(i)$.

Set $$v_{n}=Q^{2n}e^{-2n\alpha}, \ n\in \mathbb{Z}_{\geq0},$$
and set $$v_{n}^{i}=(2n-i)!Q^{i}e^{-n\alpha}+(-1)^{i}i!Q^{2n-i}e^{-n\alpha},\ i,n\in\mathbb{Z}_{\geq0},\  i< n,\  m|n-i.$$

 Then  $$\mathcal {W}(p)^{D_{m}}= \bigoplus_{n=0}^\infty \mathcal {U}(Vir)v_{n} \oplus \bigoplus_{i,n}\mathcal {U}(Vir)v_{n}^{i},$$
where the pair $(i,n)$ satisfies the above condition.
In $\mathcal {W}(p)$, we need the following  vectors:
$$H^{(2)} =Q^{2}e^{-2\alpha},$$ $$U^{(m)}=(2m)!e^{-m\alpha}+Q^{2m}e^{-m\alpha}. $$

  Note that $\Psi \in \mbox{Aut} (\overline{M(1)}$.
  So we have an automorphism of order two of the singlet vertex operator algebra $\overline{M(1)}$ such that $\Psi (H) = -H$.
Let
$$ \overline{M(1)} ^{\pm}  = \{ v \in \overline{M(1)} \ \vert \ \Psi(v) =\pm v \}. $$ Then $\overline{M(1)} ^+$ is a simple vertex subalgebra and we
clearly have
$$ \overline{M(1)} = \overline{M(1)}^+ \bigoplus \overline{M(1)}  ^-. $$
We have the following decompositions:
 $$\overline{M(1)} ^+= \bigoplus_{n=0} ^{\infty} \mathcal {U}(Vir) Q ^{2n} e ^{- 2n \alpha} \cong \bigoplus_{n=0} ^{\infty} L(c_{p,1}, h_{1, 4 n +1} ),  $$
 $$\overline{M(1)} ^-= \bigoplus_{n=0}^{\infty}\mathcal {U}(Vir) Q ^{2n+1} e ^{- (2n+1) \alpha} \cong \bigoplus_{n=0} ^{\infty} L(c_{p,1}, h_{1, 4 n +3} ),  $$

 The proofs of Theorem 3.5 in \cite{dn} and  Theorem 3.2 in \cite{a} imply the following theorem.
 \begin{theorem}\label{2b}
$\overline{M(1)} ^+$ is strongly generated by $H^{(2)}$ and $\omega$, and the Zhu algebra $A(\overline{M(1)} ^+)$ is generated by $[H^{(2)}]$ and $[\omega]$.
Moreover, we have $$[H^{(2)} ]^{2} = f([\omega])[H^{(2)}]+g([\omega])$$ for some $f, g\in\mathbb{C}[x]$, $deg(f)\leq3p-1$, $deg(g)\leq6p-2$.
\end{theorem}

We also have a uniqueness result for $\overline{M(1)} ^+$.

\begin{proposition}
 $\overline{M(1)} ^+$ is the only proper vertex subalgebra of $\overline{M(1)}$ which contains Virasoro vertex operator algebra $L(c_{p,1},0)$.
\end{proposition}

\begin{proof} Recall the $sl_2(\mathbb{C})$-action on $\mathcal{W}(p)$ from \cite{ALM}. Then
Lemma 2.6 \cite{dg}  applies verbatim for $c_{p,1}$ central charges.  The rest is clear.
\end{proof}

Our next  result concerns generators of $\triplet ^{D_m}$.
\begin{proposition}
The vertex algebra $\triplet ^{D_m}$ is strongly  generated by
$$ \omega, \ H ^{(2)}, \ U ^{(m)} = (2m )!   F ^{(m)} + E ^{(m)}, $$
so it is a $\mathcal{W}$-algebra of type $(2, 6p-2, m^2p+m(p-1) )$
\end{proposition}
\begin{proof} It is easy to prove that $\omega$,   $H^{(2)}$, $U^{(m n)}$, $n \ge 0$,  generate all of $\triplet ^{D_m}$. In fact,
$$\triplet ^{D_m}=\bigoplus_{n=0}^\infty \overline{M(1)}^+ \cdot U^{(m n)},$$
where $U^{(0)}={\bf 1}$. By using strong generation of $\overline{M(1)}^+$ via $H^{(2)}$ and $\omega$, to finish
the proof it is sufficient to show that
$\overline{M(1)}^+ \cdot U^{(m n)}$ is strongly generated by $\omega$, $U^{( m )}$ and $H^{(2)}$. We use induction to prove this fact.
Suppose that the result holds for $m \leq k $.
Then, as in proof of Theorem 2.9 in \cite{dg}, we choose $n_0$ such that
$U^{(m )}_{-n_0} U^{(m k)}=\nu U^{( m + m k )}+v$ for some $v$ for which the inductive hypothesis applies.
\end{proof}


\section{Irreducible $\overline{M(1)} ^+$--modules from untwisted $\overline{M(1)}$--modules}

The classification of irreducible $\overline{M(1)}$--modules was done in \cite{a}. The irreducible $\overline{M(1)}$--modules are parameterized by lowest weights   with respect to $( L(0), H(0))$.  So let
$L_{\overline{M(1)} } (x,y)$ denote the irreducible $\overline{M(1)}$--module with lowest weight $(x,y)$ such that
$$ y ^2  = C_p P(x), \quad  C_p = \frac{(4p)^{2p-1} } {(2p-1) ! ^2} ,  \ P(x) = \prod_{i=1} ^{2p-1} (x- h_{i,1} ), $$
where
$$h_{i,j} = \frac{(jp -i) ^2 -(p-1) ^2}{4p}. $$

 Then
$$ L_{\overline{M(1)} } (x,y) ^{\Psi} \cong  L_{\overline{M(1)} } (x,-y). $$

If $y \ne 0$, then $ L_{\overline{M(1)} } (x,y)  \cong  L_{\overline{M(1)} } (x,-y) $ are irreducible as $\overline{M(1)} ^+$--modules.

If $y = 0$, then  $ L_{\overline{M(1)} } (x,0) $ is $\Psi$-invariant, and therefore it splits into a direct sum of two non-isomorphic $\overline{M(1)} ^+$--modules:
$$ L_{\overline{M(1)} } (x,0)  \cong  L_{\overline{M(1)} } (x,0) ^+ \oplus L_{\overline{M(1)} } (x,0) ^- . $$

But in this case $x = h_{i,1} $ and $ L_{\overline{M(1)} } (h_{i,1},0)  \cong \overline{M(1)}. e ^{ \frac{i-1}{2p} \alpha}$. Therefore:
 $$ L_{\overline{M(1)} } (h_{i,1},0)  ^+ \cong  \overline{M(1)} ^+ . e ^{ \frac{i-1}{2p} \alpha}, \quad  L_{\overline{M(1)} } (h_{i,1},0)  ^- \cong    \overline{M(1)} ^+ . Qe ^{ \frac{i-1}{2p} \alpha - \alpha}. $$

For any $\lambda\in\mathfrak{b}$ let $M(1, \lambda)$ denote the irreducible $M(1)$-module defined in \cite[p.219]{ll}. Let $v_{\lambda}\neq0$ be a lowest weight vector of $M(1, \lambda)$.

The following result  follows from  \cite{am1}:
 \begin{lemma}\label{3b}
Consider $M(1, \lambda)$ as a $\overline{M(1)} ^+$-module, then
$$[H ^{(2)}](v_{\lambda})=(-1)^{p}B_{p}{t \choose 3p-1}{ t+p \choose 3p-1}v_{\lambda},$$
$$[\omega](v_{\lambda})=\frac{1}{4p}t(t-2(p-1))v_{\lambda},$$
where $B_{p}=\frac{(3p-1)!(2p)!}{(4p-1)!(p)!}$ and $t=\langle\lambda,\alpha\rangle$.
Moreover,
$$ [H ^{(2)} ] - f_p([\omega]) $$
acts trivially on top components of $M(1, \lambda),$
where
$$f_p(x) = \widetilde{B_p} \prod _{i=1} ^{3p-1} (x- h_{i,1}), \quad \widetilde{B_p} \ne 0. $$
\end{lemma}

Let us now identify lowest weights of constructed $\overline{M(1)} ^+$--modules with respect to $(L (0), H ^{(2)} (0) )$.
\begin{proposition} \label{real-1}
We have:
\item[(1)]  $$L_{\overline{M(1)} ^+ } (h_{i,1}, 0 ) =  L_{\overline{M(1)}  } (h_{i,1}, 0) ^+ \quad \mbox{for} \   i=1, \dots, p, $$
\item[(2)] $$L_{\overline{M(1)} ^+ } (x,f_p(x) ) =  L_{\overline{M(1)}  } (x,  \pm \sqrt{ C_p P(x)}) \quad \mbox{if} \   x \ne h_{i,1}, \ i = 1, \dots, p. $$
\end{proposition}
\begin{proof}
Assertion (1) follows from the fact that $L_{\overline{M(1)} ^+ } (h_{i,1}, 0 )$ is realized as a submodule of $\overline{M(1)}. e ^{ \frac{i-1}{2p} \alpha}$, and that
$$ L(0) e ^{ \frac{i-1}{2p} \alpha} = h_{i,1} e ^{ \frac{i-1}{2p} \alpha},  \quad H^{(2)} (0) e ^{ \frac{i-1}{2p} \alpha} = 0 \quad (i=1, \dots, p). $$
$$ x= \tfrac{1}{4p} t (t -2(p-1)), \ y = (-1) ^p B_p {t \choose 3p-1} {t +p  \choose 3p-1} = f_p (x). $$
For a proof of assertion (2) we realize the irreducible module $L_{\overline{M(1)} ^+ } (x, y)  $  as a irreducible subquotient of the $\overline{M(1)} ^+$--module $M(1, \lambda)$, with $\langle \lambda, \alpha \rangle = t$. Uniqueness follows from the theory of Zhu's algebras.
\end{proof}

\begin{proposition}  \label{real-3}
For every $1 \leq i \leq  p$, there exists an irreducible $\overline{M(1)} ^+$--module \newline $L_{\overline{M(1)} ^+} (h_{4p-i,1}, - 2  f_p(h_{4p-i,1} ) )$ with lowest weight
$$   (h_{4p-i,1}, - 2  f_p(h_{4p-i,1} ) ).$$
This module is realized as
$$ L_{\overline{M(1)} }(h_{i,1}, 0) ^- = \overline{M(1)} ^+ . Q e ^{ \frac{i-1}{2p} \alpha- \alpha}. $$
\end{proposition}
\begin{proof}
Let us consider the case $i=1$, i.e., the module $\overline{M(1)} ^-$.
So there is $\lambda_1 \in {\C}$ such that  $$ H ^{(2)} (0)  Q e ^{-\alpha} = \lambda_1 Q e ^{-\alpha}. $$
Then applying  the operator $f$ we get:
 $$ 6 (Qe ^{-2\alpha}) (0) Q e ^{-\alpha}  + 2 H ^{(0)} (0) e ^{-\alpha} = 2 \lambda_1 e ^{-\alpha}, $$
 which implies
 $$- 3 e ^{-2 \alpha} (0) E +  f_p(h_{4p-1,1}) e ^{-\alpha}  = \lambda_1 e ^{-\alpha}. $$
 Let $a \in {\C}$ such that $e ^{-2\alpha} (0) E = a F$. Then
 $$ H ^{(2)} (0) E = a E. $$
 Applying  the automorphism $\Psi$ we get
 $$ H^{(2)} (0) F = a F \implies a = f_{p} (h_{4p-1,1}). $$
 We get
 $$\lambda_1 = -2 f_p( h_{4p-1,1}). $$

 The general case is similar
(just change $e  ^{-\alpha} $ with $e ^{ \tfrac{i-1}{2p} \alpha - \alpha}$ and use operator $f= -\Psi_{\Lambda(i)} \circ Q \circ \Psi_{\Lambda(i)}$.
\end{proof}

\section{ Twisted $\overline{M(1)}$ and $\triplet$--modules}
\label{twisted}

In this section we shall construct a family of $\Psi$--twisted modules for the vertex operator algebras $\overline{M(1)}$, $\triplet$ and its subalgebras $\triplet ^{A_m}$, $m \ge 1$. We shall start from $\tau$--twisted modules constructed in \cite{ALM}, and deform this action to get $\Psi$--twisted modules.

First we recall one well-known fact (see \cite{DM-Galois} for details). Let $V$ be a vertex operator algebra, $\tau$ and $\sigma$ automorphisms of $V$. Assume that $(M,Y_M)$ is a $\tau$-twisted $V$--module.
Then $(M^{\sigma}, Y_M ^{\sigma})$ is a $ \sigma ^{-1} \tau \sigma $-twisted module, where
$$ M ^{\sigma} = M; \quad Y_M ^{\sigma} (v, z) = Y_M (\sigma v, z). $$

Consider now twisted $\triplet$--modules $\overline{R(j)} = V_{L + \frac{3p -j - 1/2}{2p} \alpha}$, $j=1, \dots, 2p$, studied
previously in \cite{ALM}. These modules are $\tau$--twisted $\triplet$--modules, where  $\tau = \exp[ \tfrac{\pi i}{2p} \alpha(0)]$ is the automorphism of $\triplet$ of order $2$ (for more details see \cite{ALM}).

We have:
\begin{proposition}
Let $ \sigma \in \mbox{Aut}(\triplet) = PSL(2, \C). $
 Then $\overline{R(j)} ^{\sigma} $ is $\sigma^{-1} \tau \sigma$--twisted $\triplet$--module.
Assume that $\mathcal{W}_0$ is a subalgebra of $\triplet$ such that  $$\sigma \tau \sigma ^{-1} = 1  \quad \mbox{on} \ \ \mathcal{W}_0. $$ Then $\overline{R(j)} ^{\sigma}$ is an untwisted $\mathcal{W}_0$--module.
\end{proposition}

Take now the following automorphism of the triplet vertex algebra
$$ \sigma = \exp[ \mu f ] \exp[\lambda Q], \ \lambda \mu = -1/2. $$
Then
\bea && \sigma (H) =  \lambda E + \mu  F.  \nonumber \\
&& \sigma (E) =  \frac{2 \lambda ^2 E - 2 \lambda H + F}{2 \lambda ^2} \nonumber \\
&& \sigma (F) = \frac{2 \lambda ^2 E + 2 \lambda H + F}{4} \nonumber
\eea
Let $U = \sigma(H)$. Then $U$ generates the subalgebra of $\triplet$ isomorphic to $\overline{M(1)}$.

Then $$\tau (\sigma (H)) = \tau (U) = - U = - (\sigma(H)), $$
and
$$  \sigma ^{-1} \tau (\sigma (E)) = - \frac{2}{\lambda^2} F,  \quad \sigma ^{-1} \tau \sigma (F) = - \frac{\lambda ^2}{2}  E, \quad \sigma ^{-1} \tau \sigma (H) = - H. $$

Take $\lambda =   i$. Then

$$\sigma ^{-1} \tau \sigma = \Psi. $$

\begin{proposition}
For $j = 1, \dots, 2 p$, we have

\item[(1)] $\overline{R(j)} ^{\sigma}$  is  a $\Psi$--twisted $\triplet$--module.

\item[(2)] $\overline{R(j)} ^{\sigma}$  is  a $\Psi$--twisted  $\overline{M(1)}$--module.

\item[(3)] $\overline{R(j)} ^{\sigma}$ is  an (untwisted) $\overline{M(1)} ^+$--module.

\item[(4)] $\overline{R(j)} ^{\sigma}$ is an (untwisted) $\triplet ^{D_{m}}$--module for every $m \ge 1$.
\end{proposition}

We shall here discuss the $m=1$ case (this case was also discussed in \cite{F}). We have:

\begin{proposition} \label{m1}
We have:
$$ \triplet^{D_1} \cong \triplet ^{A_2}. $$
In particular, $ \triplet ^{D_1}$ is a $C_2$--cofinite vertex operator algebra of type $(2,2p-1,6p-2)$.
\end{proposition}
\begin{proof}
We have
$$ v \in \triplet ^{D_1} \iff \psi (v) = v \iff \sigma ^{-1}  \tau \sigma (v) = v \iff \sigma(v) \in \triplet ^{A_2}. $$
This implies that
$$ \sigma \vert _{\triplet ^{D_1}} : \triplet ^{D_1} \rightarrow \triplet ^{A_2}$$
is an isomorphism of vertex operator algebras.

Since $\triplet ^{A_2}$ is $C_2$--cofinite (cf.\cite{ALM}), we have that $\triplet ^{D_1}$ is $C_2$--cofinite.
\end{proof}

We shall use $\triplet^+$ instead of $\triplet^{D_1}$.

\section{The classification of  twisted $\overline{M(1)}$--modules}

In this section we study the problem of classification of irreducible $\Psi$--twisted $\overline{M(1)}$--modules. Provided that some technical conditions are satisfied (which can be verified for $p$ small on computer), we completely described twisted Zhu's algebra $A_{\Psi}(\overline{M(1)})$, which gives the classification of irreducible twisted modules.  This section is of independent interest.

\bigskip
Assume that $\sigma  \in \mbox{Aut} (\triplet)$ such that
\bea &&\label{cond-1} \sigma  ^{-1} \tau \sigma   = \Psi \quad \mbox{on} \ \overline{M(1)}  \\
&& \label{cond-2}  \sigma  (H) = \lambda E + \mu F, \quad (\lambda, \mu \ne 0). \eea
Then $\overline{R(j)} ^{\sigma }$ is irreducible $\sigma  ^{-1} \tau \sigma$--twisted $\triplet$--modules. But these modules can be treated as $\Psi$--twisted $\overline{M(1)}$--modules.

General form of $\Psi$--twisted $\overline{M(1)}$--modules can be obtained by using the theory of twisted Zhu's algebras.

Recall that for an automorphism $\tau$ of order two of the vertex operator algebra $V$, twisted Zhu's algebra $A_{\tau}(V) = V / O_{\tau} (V)$ where $O_{\tau} (V)$ is the linear span of the vectors of the form
$$ a  {\circ} ^{\tau}   b = \mbox{Res}_z \frac{(1+z) ^{\deg (a) -1 + \delta_r + r/2} }{z ^{1+ \delta_r} } Y(a, z) b $$
where $a \in V ^{r}$, $\delta_0 = 1$, $\delta_1 = 0$, $r =0,1$.

We also define
$$ a {\circ}_k ^ {\tau}  b = \mbox{Res}_z \frac{(1+z) ^{\deg (a) -1 + \delta_r + r/2} }{z ^{k + \delta_r} } Y(a, z) b  \quad (k \ge 1).$$
We know that $a {\circ}_k ^ {\tau}  b \in O_{\tau} (V)$.
\begin{lemma}
There is a surjective   homomorphism of associative algebras
$$\Phi: A(L(c_{p,1},0) ) \rightarrow A_{\Psi} (\overline{M(1)}).$$
In particular, the $\Psi$--twisted Zhu's algebra $A_{\Psi} (\overline{M(1)})$ is a quotient of the polynomial algebra ${\Bbb C}[x]$.
 Moreover, every irreducible $\Psi$--twisted $\overline{M(1)}$--modules has the form $L ^{tw} _{ \overline{M(1)} } (h)$ such that the lowest component is $1$--dimensional on which  $[\omega]$ acts by multiplication with $h$.
\end{lemma}
\begin{proof}
Since $\overline{M(1)}$ is strongly generated by $\omega$ and $H$, then
$A_{\Psi} (\overline{M(1)})$ is generated by
$$[\omega] = \omega + O_{\Psi} (\overline{M(1)}), \ \ [H]= H + O_{\Psi} (\overline{M(1)}). $$
But, $H = H \circ {\bf 1}$ implies that $A_{\Psi} (\overline{M(1)})$ is generated by $[\omega]$. This gives an algebra homomorphism
$$ A(L(c_{p,1},0) \cong {\Bbb C}[x] \rightarrow  A_{\Psi} (\overline{M(1)}). $$
The second assertion follows from the theory of Zhu's algebras.
\end{proof}

 We have:

\begin{theorem} \label{ired-2}
For every $j =1, \dots, p$:
\item[(1)] $\overline{R(j)} ^{\sigma}$ is an irreducible $\Psi$--twisted $\overline{M(1)}$--modules of lowest conformal weight $h_{p+1/2-j,1}$.
\item[(2)]
$ \overline{R(j)} ^{\sigma} \cong \overline{R(2p +1 -j)} ^{\sigma} \cong L ^{tw} _{ \overline{M(1)} } (h_{p+1/2-j,1} ). $
\item[(3)] $\overline{R(j)} ^{\sigma} \cong  (\overline{R(j)} ^{\sigma} ) ^+ \bigoplus (\overline{R(j)} ^{\sigma} )^-$,
and $(\overline{R(j)} ^{\sigma} ) ^{\pm}$ are irreducible $\overline{M(1)} ^+$--modules.
\end{theorem}
\begin{proof}
First we notice that as a module for the Virasoro algebra  $\overline{R(j)} ^{\sigma}$ is a direct sum of irreducible Feigin-Fuchs modules. Irreducibility follows easily from the fact that $\sigma(H) = \lambda E + \mu F $ maps Feigin-Fuchs module
$$e ^{ \frac{ 3p -1/2 - s}{2p} \alpha} \otimes M(1) \rightarrow e ^{ \frac{ p -1/2 - s}{2p} \alpha} \otimes M(1) \bigoplus e ^{ \frac{  5 p -1/2 - s}{2p} \alpha} \otimes M(1) \quad (s \in {\Z} ). $$
This proofs of irreducibility of $\overline{R(j)} ^{\sigma}$. The assertion (2) follows from (1), Zhu's algebra theory  and the fact that  $e ^{ \frac{ p -1/2 -j}{2p} \alpha}$  (resp. $e^{\frac{p - 3/2 + j}{2p} \alpha}$ ) is a lowest weight vector in $\overline{R(j)} ^{\sigma}$ (resp. $\overline{R(2p+1-j)} ^{\sigma}$).

Assertion (3) follows directly from (1) and (2).
\end{proof}

 \begin{conjecture} \label{conj-twisted}
 Twisted Zhu's algebra $A_{\Psi}( \overline{M(1)} )$ is isomorphic to
 $$ {\Bbb C}[x] / \langle h_{p} (x)  \rangle, \quad h_p(x) = \prod_{i=1} ^p  (x - h_{p+1/2-i,1}). $$
 \end{conjecture}

We shall now see that Conjecture \ref{conj-twisted} holds for $p$ small.
We notice that $$H \circ ^{\tau} H, H \circ ^{\tau} _3 H \in U(Vir). {\bf 1} = L(c_{p,1},0). $$ Therefore,  there are polynomials
 $f, g \in {\Bbb C}[x]$ such that
 $$  [ H \circ ^{\tau} H  ] = f([\omega]), \quad  [H \circ ^{\tau} _3 H] = g ([\omega]) \ \in A(L(c_{p,1},0) ) \cong {\Bbb C}[x]. $$
 But in twisted Zhu's algebra $A_{\Psi}(\overline{M(1)} )$ we have
 $$ f([\omega]) = g([\omega]) = 0. $$
 Since $\overline{R(j)} ^{\sigma}$ are irreducible $\Psi$-twisted $\overline{M(1)}$--modules, we conclude that
 $$ f( h_{p+1/2-j,1} ) = g( h_{p+1/2-j,1} ) = 0 \quad (j=1, \dots,p).$$
 Therefore,
 $f, g$ are divisible with $h_p$, and there exists polynomials $f_1, g_1$ such that
 $$ f= f_1 h_p, \quad g =g_1 h_p. $$
 It is clear that the proof of Conjecture \ref{conj-twisted} will follow from
 \bea
 && \label{rel-prime} (f_1, g_1) = 1.
 \eea

 The polynomials $f,g $ can be determined by evaluating above relations on $M(1)$--modules $M(1, \lambda)$.  Since such construction appear in many paper, we shall omit some details and only present formulas:
 \bea
 f(x) & = & \mbox{Res}_{z_1, z_2, z_3} \frac{(1+z_1) ^{ 2 p  - t - 3/2} (1+z_2) ^t (1+z_3) ^t  }{z_1 ^{-2p+1} (z_2 z_3)^{2p} }  \nonumber \\  &&  \left((z_1-z_2) ^{-2p} (z_1-z_3) ^{-2p}(z_2-z_3) ^{2p}    - (z_2-z_1) ^{-2p} (z_1-z_3) ^{-2p}(z_2-z_3) ^{2p} \right), \nonumber
 \eea
  \bea
 g(x) & = & \mbox{Res}_{z_1, z_2, z_3} \frac{(1+z_1) ^{ 2 p  - t - 3/2} (1+z_2) ^t (1+z_3) ^t  }{z_1 ^{-2p+3} (z_2 z_3)^{2p} }  \nonumber \\  &&  \left((z_1-z_2) ^{-2p} (z_1-z_3) ^{-2p}(z_2-z_3) ^{2p}    - (z_2-z_1) ^{-2p} (z_1-z_3) ^{-2p}(z_2-z_3) ^{2p} \right) \nonumber
 \eea
 where $x= \frac{t (t-2p+2)}{4p}$.

 It turns out that
  formulas for polynomials $f_1, g_1$ are very complicated.

 By using Mathematica/Maple we get a list of polynomials $f_1(x)$ and $g_1(x)$ for $p \le 5 $ is (up to a scalar factor):

\vskip 5mm

{\tiny
\begin{center}
  \begin{tabular}{|c|c|c|c}
    \hline
  $ \ \  p$  & $f_1(x) $ & $ g_1(x) $ \\ \hline \hskip 2mm
 $2 $& $  8 x -3  $ & $32 x - 9$ \\ \hline
  $ \ \ 3 $  & $ 105 - 256 x + 256 x^2 $ &  $455 - 1136 x + 1536 x^2 $ \\
   \hline
   $\  \ 4  $  & $  -23625 + 68960 x - 59392 x^2 + 32768 x^3  $ & $-401625 + 1216000 x - 1024000 x^2 + 786432 x^3
   $ \\
   \hline
   $ \ \ 5 $  & $ 2837835 - 9007488 x + 10473984 x^2 - 4587520 x^3 + 1638400 x^4 $ & $ 3972969 - 12744144 x + 15591168 x^2 - 6410240 x^3 + 3276800 x^4   $ \\
    \hline
  \end{tabular}
\end{center}
 }
 \vskip 5mm

We checked by computer that all polynomials $f_1, g_1$ are relatively prime for $p \le 20$.

\begin{theorem} Assume that condition (\ref{rel-prime}) holds (verified for $p$ small). Then Conjecture \ref{conj-twisted} holds and
  the set  $$\{ L_{\overline{M(1)} } ^{tw} (h_{p + 1/2-j ,1}) \ \vert \ j=1, \dots, p  \}$$ provides all irreducible, $\tfrac{1}{2}\N$--graded $\Psi$--twisted $\overline{M(1)}$--modules.

\end{theorem}
\begin{proof}
We know that $\mbox{Ker}(\Phi)$ is a principal ideal in ${\Bbb C}[x]\cong A(L(c_{p,1},0))$ generated by certain polynomial $\widehat{h}([\omega])$ such that $h_p \vert \widehat{h}$.  Since $f([\omega]), g([\omega]) \in \mbox{Ker}(\Phi)$, condition (\ref{rel-prime}) implies that $h_p([\omega]) \in \mbox{Ker}(\Phi)$. This proves the assertion.
\end{proof}

\begin{remark}
It was proved by C. Dong and C. Jiang in \cite{DJ-tams} that semi-simplicity of twisted Zhu's algebra $A_g(V)$ implies $g$ rationality. By using this result we conclude that every $\Psi$--twisted $\tfrac{1}{2} \N$--graded $\overline{M(1)}$--module is completely reducible.
\end{remark}

\section{$\overline{M(1)} ^+$ and $\triplet ^{D_m}$--modules from twisted modules}

We have the following $ 2 p $ irreducible $\overline{M(1)} ^+$--modules
$$ R(j) ^{\sigma} := (\overline{R(j)} ^{\sigma} ) ^{-} \cong (\overline{R(2 p +1 -j)} ^{\sigma} ) ^{-},  $$ $$  R(2p + 1-j ) ^{\sigma}:= (\overline{R(j)} ^{\sigma} ) ^+ \cong (\overline{R(2p + 1 - j)} ^{\sigma} ) ^{+}. $$
These modules are irreducible $\overline{M(1)} ^+$--modules,  and therefore they are irreducible $\triplet ^{D_m}$--modules for every $m \ge 1$.

First we shall identify these modules as $\overline{M(1)} ^+$--modules.

\begin{proposition} \label{real-2}
For every $1\leq i  \leq  2p$, there exists an irreducible $\overline{M(1)} ^+$--module \\
$L_{\overline{M(1)} ^+}  (h_{3p+1/2-i,1}, -\frac{1}{2} f_p(h_{3p+1/2-i, 1} ) )$ with lowest  weight
$$ (h_{3p+1/2-i,1}, -\frac{1}{2} f_p(h_{3p+1/2-i, 1} ) ).  $$
Moreover,
 $$ {R(i)} ^{\sigma} \cong  L_{\overline{M(1)} ^+}  (h_{3p+1/2-i,1}, -\frac{1}{2} f_p(h_{3p+1/2-i, 1} ) ). $$
\end{proposition}

\begin{proof}
First  we notice that the vector $e ^{ \frac{3p-1/2-j}{2p} \alpha}$ is a lowest weight vector in $ {R(j)} ^{\sigma}$ for $\overline{M(1)} ^+$.
The action of generator $H ^{(2)}$ is
$$ \sigma (H ^{(2)}) = -\frac{1}{2} H ^{(2)} + a_1 F ^{(2)}   + a_2 E ^{(2)}$$
for certain constants $a_1, a_2$.
This implies that lowest weight  with respect to $(L(0), H ^{(2)} (0) ) $ is
$$ (h_{3p+1/2-j,1}, -1/2 f_{p} (h_{3p+1/2-j,1})). $$
The proof follows.
\end{proof}

\begin{remark}
{\em It is known that for $p=1$, the lattice vertex algebra $V_L$ carries an action of the $A_1^{(1)}$ affine Lie algebra. In Chapters 3 and 4 of \cite{FLM}, Frenkel, Lepowsky and Meurman discuss twisted and untwisted constructions of the affine Lie algebra $A^{(1)}_1$. Two twisted realizations of $\hat{\frak{a}}[\theta_1]$ and $\hat{\frak{a}}[\theta_2]$ are constructed there, where $\theta_1$ and $\theta_2$ are two obvious involutions of $A_1$.  In Section 3  (loc.cit) (essentially Lepowsky-Wilson's construction), $\hat{a}[\theta_2]$ acts on the twisted Heisenberg Fock space $S(\frak{h}_{\mathbb{Z}+1/2})$, and then in Section 4 a different construction is presented on a pair of $\theta_1$-twisted $V_L$-modules. Finally, a graded isomorphism of two $A_1^{(1)}$-modules is constructed by using automorphism $\sigma$ (for definition see Chapter 4 (loc.cit)).

But in this $\mathcal{W}$-algebra picture, where $p \geq 2$, there is no such thing as $S(\frak{h}_{\mathbb{Z}+1/2})$, so instead we have to reverse the steps. We use the automorphism $\tau$ (corresponding to $\theta_1$), $\tau$-twisted modules $\overline{R(j)}$ defined in \cite{ALM}, and an automorphism $\sigma$ constructed in Section 4, to  {\em define}  $\sigma^{-1} \tau \sigma$-twisted action of $\triplet$ on the same space. This in turn gives a family of irreducible $\overline{M(1)} ^+$--modules needed in Proposition \ref{real-2}.
Thus our construction should be viewed as a $\mathcal{W}$-generalization
of the relevant parts of \cite{FLM}.}
\end{remark}

\begin{lemma}
There is an automorphism $h$ of $\triplet ^{D_m}$ such that
$$ h( U ^{(m)} ) = - U ^{(m)}, \quad h( H ^{(2)} ) = H ^{(2)}. $$
\end{lemma}
\begin{proof}
The automorphism $h$ is realized as a restriction of the automorphism $$\exp[ \frac{\pi i}{   2 m p } \alpha(0)] $$
of the lattice vertex algebra $V_L$.
\end{proof}

\begin{proposition} \label{r-sigma-modules}
For every $ 1  \le j \le 2p$, $R(j) ^{\sigma}$ and  $R(j) ^{h \sigma}$ are irreducible $\triplet ^{D_m}$--modules with lowest weights
$$ (h_{3 p + 1/2 - j, 1}, - 1/2 f_{p} (h_{3 p + 1/2 - j, 1}), a_j), \ (h_{3 p + 1/2 - j, 1}, - 1/2 f_{p} (h_{3 p + 1/2 - j, 1}), - a_j),$$
where $a_j \ne 0$. In this way we have constructed $4p$ non-isomorphic irreducible $\triplet ^{D_m}$--modules.
\end{proposition}
 \begin{proof}
 Theorem \ref{ired-2} gives that $R(j)^{\sigma}$ is  $\overline{M(1)} ^+$--modules, and  therefore it is  also irreducible $\triplet ^{D_m}$--module. Lowest weight component of these modules is $1$--dimensional,  spanned by vector $v$, and the  lowest weight with respect to $(L(0), H ^{2} (0), U ^{m} (0))$
$$ (h_{3p+1/2-j,1}, - 1/2 f_p( h_{3p+1/2-j,1}), a_j)     \quad \mbox{for} \ R(j) ^{g },$$
where $a_j \in {\C}$.   Since $Y(U ^{(m)}, z) v \ne 0$, we have that there is $n_0$ such that
$$ U^{(m)}_{n_0} v \ne 0, \quad U^{(m)}_{n} v = 0 \ (n > n_0). $$
One can easily see that $ U^{(m)}_{n_0} v$ is a singular vector for $\overline{M(1)} ^+$. Irreducibility of $R(j)^{\sigma}$ as $\overline{M(1)} ^+$--module gives that $ U^{(m)}_{n_0} v$ is proportional to $v$, and therefore
$ U^{(m)}_{n_0} v = a_j v$ and $a_j \ne 0$.
Applying the automorphism $h$ of $\triplet^{D_m}$ we get that $R(j) ^{h \sigma}$ is irreducible lowest weight module with lowest weight
$$ (h_{3p+1/2-j,1}, - 1/2 f_p( h_{3p+1/2-j,1}), -a_j). $$
The proof follows.
\end{proof}

\section{Classification of irreducible $\overline{M(1)} ^+$--modules }

In this section we study the problem of classification of $\overline{M(1)} ^+$--modules by using Zhu's algebra theory. We will prove that every irreducible $\overline{M(1)} ^+$--modules  is realized as a submodule of untwisted $\overline{M(1)}$--modules   or  $\Psi$--twisted  $\overline{M(1)}$--modules.

  For $a, b \in \triplet$ we define

 $$ a {\tilde{\circ}}  b = \mbox{Res}_z \frac{(1+z) ^{\deg (a)} } {z ^3} Y(a,z) b , $$
 $$ a {\tilde{\circ}}_k b = \mbox{Res}_z \frac{(1+z) ^{\deg (a)} } {z ^k} Y(a,z) b. $$

 \begin{lemma} \label{rel-sing} Inside the Zhu algebra $A ( \overline{M(1)} ^+) $ we have the following relations:

 \item[(1)] $$ ([H^{(2) }] - f_p([\omega]) ) * ([H ^{(2)}] - r_p([\omega]) ) = 0,$$
 \item[(2)] $$\ell_p([\omega]) * ( [H ^{(2)}] - f_p([\omega] ) ) = 0,$$
 where $r_p , \ell_p \in {\C}[x]$, $\deg r_p \le 3p-1$, $\deg \ell_p \le 3p. $
 \end{lemma}

\begin{proof}By using Lemma  \ref{3b} we get that
if $R(x,y) \in {\C}[x,y]$ such that
$$ R([\omega], [H ^{(2)}]) = 0 \implies  (y  - f_p(x) ) \ \mid R(x,y). $$
Now assertion (1) follows from relation
$$[H ^{(2)}] * [H ^{(2)}] - f([\omega]) * [H ^{(2)}] - g([\omega]) =0,
 $$
 where $\deg(f) \le 3p-1$, $\deg (g) \le 6p-2. $
Assertion (2) follows from relation
$$0 = [H ^{(2)} \tilde{\circ} H ^{(2)}] = \tilde{f} ([\omega]) [H^{(2)}] + \tilde{g} ([\omega]) =0, $$
where $\deg (\tilde{f}) \le 3p$, $\deg(\tilde{g}) \le 6 p-1$.
\end{proof}

Irreducible representations of $\overline{M(1)} ^+$ are parameterized by lowest weight $(x,y)$ with respect to $(L(0), H ^{(2)} (0))$. Lemma \ref{rel-sing} and Theorem \ref{2b} imply the following result.
\begin{proposition} \label{clas-1}
 Let $L_{\overline{M(1)} ^+ } (x,y)$ be irreducible $\overline{M(1)} ^+$--module of lowest  weight $(x,y)$. Then
$$(x,y) \in \mathcal{S}_p = \mathcal{S}_p ^{(1) } \cup \mathcal{S}_p ^{(2)},$$
where
\begin{eqnarray}
 \mathcal{S}_p ^{(1) }& = &\{ (x,y) \in {\C} ^2 \ \vert \ y  = f_p(x) \} ,
\nonumber  \\
 \mathcal{S}_p ^{(2) }&=&  \{ (x,y) \in {\C} ^2 \ \vert \ \ell_p(x) = 0, \ y  = r_p(x) \}.
 \end{eqnarray}
\end{proposition}

By evaluating relations (1)  and  (2) in Lemma \ref{rel-sing}  on $\overline{M(1)} ^+$--modules, we get the following formulas for polynomials $\ell_p(x)$ and $r_p(x)$.
\begin{corollary} \label{interp}
\item[(1)] There is a constant $A_p$ (possibly zero) such that
$$ \ell_p(x) = A_p \prod_{i=1} ^{p}(x- h_{4p-i,1}) \prod_{i=1} ^{2p} (x-h_{3p+1/2 -i,1}). $$

\item[(2)] The polynomial $r_p$ is non-trivial and satisfies the following interpolation conditions:
$$ r_p(h_{4p-i,1} ) = -2 f_p (h_{4p-i,1} ), \qquad r_{p} (h_{3p+1/2-j,1}) = - \tfrac{1}{2} f_p (h_{3p+1/2-j,1}), $$
for $i=1, \dots, p$, $j=1, \dots, 2p$.
\end{corollary}

The following Corollary is a direct consequence of the proof Lemma \ref{rel-sing} and it is important for determination of
$C_2$--algebra $\mathcal{P} (\overline{M(1)} ^+) = \overline{M(1)} ^+ / C_2( \overline{M(1)} ^+). $
\begin{corollary} \label{struktura-singlet}
In $\mathcal{P} (\overline{M(1)} ^+)$ we have that

\item[(i)] $$( \overline{H ^{(2)} }  - a \overline{\omega} ^{3p-1} )  ( \overline{H ^{(2)} }  - b \overline{\omega} ^{3p-1} )= 0,  $$

\item[(ii)]  $$A_p \overline{\omega} ^{3p} ( \overline{H ^{(2)} }  - a \overline{\omega} ^{3p-1} ) =   0,  $$
where $a, b \in {\C}$, $a \ne 0$.
\end{corollary}

Now, we develop a different  method for calculation  of polynomials $r_p$ and $\ell_p$. We should say that in the case $p=1$, these polynomials are calculated by C. Dong and K. Nagatomo in \cite{dn}.
We shall follow our methods for  developing calculations in Zhu's algebra for triplet vertex algebras from \cite{am1}  and \cite{am2}.

\begin{lemma} \label{rel-sing-2}
Let
$f(x) = f_p(x) + r_p(x)$.
In Zhu's algebra $A(\overline{M(1)})$ we have:
\item[(1)] $$f ([\omega]) * [H ^{(2)} ] = [ Q ^2 ( Q e^{-2\alpha} *  Q e ^{-2 \alpha} )], $$
\item[(2)] $$ \ell_p ([\omega]) * [H ^{(2)} ] = [ Q^2 ( Q e ^{-2 \alpha} {\tilde{\circ}} Q e ^{-2\alpha} )]. $$
\end{lemma}
\begin{proof}
First we notice that
$$ [H ^{(2)}] * [H ^{(2)}] = f([\omega] )  * [H ^{(2)} ] - f_p([\omega]) * r_p([\omega]). $$
Applying the operator $ Q f$ on this relation, we get relation (1).
Similarly from
$$ 0 = [H ^{(2)} {\tilde{\circ}} H ^{(2)}]  = \ell_p([\omega]) * [H ^{(2)} ]- \ell_p([\omega]) * f_p([\omega])  \in A (\overline{M(1)} ^+),$$
applying again the operator $Q f$, we infer relation (2).
\end{proof}

We shall evaluate these relations on top components of lowest  weight $M(1)$--modules. Let
$$ u = Q ^2 ( Q e ^{-2\alpha} * Q e ^{-2\alpha} ) = - Q ^2 ( e ^{-2 \alpha} * Q ^2 e ^{-2\alpha}), $$
$$ \tilde{u} = Q^2 ( Q e ^{-2 \alpha}  {\tilde{\circ} }  Q e ^{-2 \alpha} ) = Q ^3 ( e ^{-2 \alpha} {\tilde{\circ}} Q e ^{-2 \alpha} ) -
 Q ^2 ( e ^{-2 \alpha} \tilde{\circ} Q ^2 e ^{-2\alpha}). $$

 Next we get
 \bea
   && o(u) v_{\lambda} = G_p(t) v_{\lambda}  \nonumber \\
     G_p(t) & = & \mbox{Res}_{z, z_1, z_2, z_3, z_4}  (  \frac{(1+z) ^{6p-2-2t} } {z^{-8p+1} (z_1 z_2 z_3 z_4) ^{4p} } \ \Delta_4 (z_1,z_2, z_3, z_4) ^{2p} \nonumber \\
  && (  (z_1-z) (z_2-z) (z-z_3) (z-z_4) ) ^{-4p}  (1+z_1) ^t (1+z_2) ^t (1+z_3) ^t (1+z_4) ^t ).   \nonumber
 \eea

 \bea
   && o(\tilde{u}) v_{\lambda} = \widetilde{G}_p(t) v_{\lambda}  \nonumber \\
      \widetilde{G}_p (t) &=&
    \mbox{Res}_{z, z_1, z_2, z_3, z_4}  (  \frac{(1+z) ^{6p-2-2t} } {z^{-8 p+3} (z_1 z_2 z_3 z_4) ^{4p} } \ \Delta_4 (z_1,z_2, z_3, z_4) ^{2p} \nonumber \\
  && (  (z_1-z) (z_2-z) (z_3-z) (z-z_4) ) ^{-4p}  (1+z_1) ^t (1+z_2) ^t (1+z_3) ^t (1+z_4) ^t )  \nonumber \\
     &&
    - \mbox{Res}_{z, z_1, z_2, z_3, z_4}  (  \frac{(1+z) ^{6p-2-2t} } {z^{-8p+3} (z_1 z_2 z_3 z_4) ^{4p} } \ \Delta_4 (z_1,z_2, z_3, z_4) ^{2p} \nonumber \\
  && (  (z_1-z) (z_2-z) (z-z_3) (z-z_4) ) ^{-4p}  (1+z_1) ^t (1+z_2) ^t (1+z_3) ^t (1+z_4) ^t )   \nonumber \\
&&  = \mbox{Res}_{z, z_1, z_2, z_3, z_4}  (  \frac{(1+z) ^{6p-2-2t} } {z^{-8 p+3} (z_1 z_2 z_3 z_4) ^{4p} } \ \Delta_4 (z_1,z_2, z_3, z_4) ^{2p} \frac{\partial_{z}^{4p-1}}{(4p-1)!} z_3^{-1}  \delta\left(\frac{z}{z_3}\right)\nonumber \\
  && \cdot ((z_1-z) (z_2-z) (z-z_4) ) ^{-4p}  (1+z_1) ^t (1+z_2) ^t (1+z_3) ^t (1+z_4) ^t ),  \nonumber
 \eea
where $\delta(x)=\sum_{n \in \mathbb{Z}} x^n$ is the formal $\delta$-function.

 Polynomial $\ell_p$ can be obtained by using the following constant term identity:
 \begin{conjecture}  \label{slutnja-1}
 $$ \widetilde{G}_p (t)  = A_p { t + p + 1/2 \choose 4p } { t + 2p \choose 4p -1} { t  \choose 4p -1}, \quad A_p \ne 0. $$
 \end{conjecture}

\begin{proposition}
Assume that Conjecture \ref{slutnja-1} holds, then
$$ \ell_p(x) = \nu \prod_{i=3p} ^{ 4p-1} (x-h_{i,1}) \cdot \prod_{i=1} ^{2p} (x-h_{3p+1/2 -i,1} ) \quad (\nu \ne 0). $$
\end{proposition}

\begin{remark}
We checked Conjecture \ref{slutnja-1} by using Mathematica for $p \le 10$.
\end{remark}

\begin{example}
Let $p=1$. Using Mathematica/Maple we get
$$ r_1(x) = -\frac{ 908 x ^2 - 515 x +27}{105}. $$
In this way, we have reconstructed the result from \cite{dn} (we take slightly different normalization).

For $p=2$ we get the following formula:

$$r_2(x) =\frac{8 (-8505 - 16875 x + 655191 x^2 - 1359879 x^3 + 642800 x^4 +
   10048 x^5)}{984555}  , $$



\end{example}

\begin{theorem}
Assume that Conjecture \ref{slutnja-1} holds. The modules constructed in Propositions \ref{real-1}, \ref{real-2} and \ref{real-3} provide a complete list of irreducible $\overline{M(1)} ^+$--modules.
\end{theorem}
\begin{proof}
Assume that $U_{x_0,y_0}$ is an irreducible $\overline{M(1)}^+$--module with lowest weight $(x_0,y_0)$. If  $(x_0,y_0) \in \mathcal{S}_p ^{(1)}$, by Proposition \ref{real-1} we know that there is a unique irreducible module with this lowest weight.

If  $(x_0,y_0) \notin \mathcal{S}_p ^{(1)}$, then $\ell_p( x_0)  = 0$. This implies that
$ x_0 = h_{4p-i,1} $ for certain $1 \le i \le p$ or $x_0 = h_{3p+1/2-i,1}$ for certain $ 1 \le i \le 2p$. Then $y_0 = r_p(x_0)$.  Therefore $(x_0,y_0) \in \mathcal{S}_p ^{(2)}$. The proof follows.
\end{proof}

\section{The vertex algebra $\triplet ^{D_2}$ }
\label{section-d2}

In this section we shall present a complete results on classification of irreducible modules, the structure of Zhu's algebra and $C_2$--algebra for the vertex operator algebra $\triplet ^{D_2}$. It turns out that many relations obtained in the case of the vertex operator algebra $\overline{M(1)} ^+$ can be    used directly.

We first recall that the vertex algebra $\triplet ^{D_2}$ is strongly  generated by
$$ \omega, \ H ^{(2)}, \ U ^{(2)} = 24  F ^{(2)} + E ^{(2)}.  $$

Therefore Zhu's algebra $A(\triplet ^{D_2})$ is generated by $[\omega], [H ^{(2)}] $ and $[U ^{(2)}]. $

Next we need  the construction of an automorphism of $\triplet ^{D_2}$.
\begin{lemma} \label{constr-aut}
There exists an automorphism  of the vertex operator algebra $\triplet ^{D_2}$ such that
$$ g( H ^{(2)} )= -\frac{1}{2}  H ^{(2)} - \frac{1}{8} U ^{(2)}, $$
$$ g (U ^{(2)} ) = 6 H ^{(2)} - \frac{1}{2} U ^{(2)}. $$
\end{lemma}
\begin{proof}
One can directly show that
$$ g = \alpha \vert_{\triplet ^{D_2}} $$
where $\alpha \in  \mbox{Aut} (\triplet)$ , $$\alpha   = \big[(\begin{smallmatrix} a  & -\frac{1}{2a} \\   a&  \frac{1}{2a} \end{smallmatrix})\big]\in PSL(2,\mathbb{C})$$
where complex number $a$ is such that $4 a ^4 = -1$. Proof follows.
\end{proof}

\begin{remark}
The automorphismsi $g$ and $h$ generates $S_3$. Therefore, $S_3$ is a subgroup of the full automorphism group of $\triplet ^{D_2}$. In fact, by using the similar proof as in \cite{dg}, one can show that
$$ \mbox{Aut} (\triplet ^{D_2}) \cong S_3. $$
\end{remark}

Let $A_0 (\triplet ^{D_2} )$ denote the subalgebra of Zhu's algebra $A(\triplet ^{D_2})$ generated by $[\omega]$ and $[H ^{(2)}]$.
First we notice that in $A(\triplet ^{D_2})$
$$\frac{1}{24} [ U ^{(2)}] * [U ^{(2)}] + 2 [H ^{(2)}] * [H ^{(2)}] + 4 [Q^2 (Q e ^{-2\alpha} * Q e ^{-2\alpha})] = 0 $$
which implies that
$$ [ U ^{(2)}] * [U ^{(2)}] \in A_0 (\triplet ^{D_2} ). $$
Thus we have
$$ A(\triplet ^{D_2}) = A_0 (\triplet ^{D_2} ) \oplus A_2 (\triplet ^{D_2} ), $$
where
$$ A_2 (\triplet ^{D_2} ) = A_0 (\triplet ^{D_2} ). [U ^{(2)}]. $$
\vskip 5mm

Similarly we have decomposition:
$$ \mathcal{P} (\triplet ^{D_2})  =  \mathcal{P}_0 (\triplet ^{D_2}) \oplus \mathcal{P}_2 (\triplet ^{D_2})$$
where
$ \mathcal{P}_0 (\triplet ^{D_2})$ is the subalgebra of  $\mathcal{P}(\triplet ^{D_2})$ generated by $\overline{\omega}$ and $\overline{H^{(2)}}$ and
$$ \mathcal{P}_2 (\triplet ^{D_2}) =  \mathcal{P}_0 (\triplet ^{D_2}) . \overline{U ^{(2)}}. $$

\begin{lemma} \label{lem-d2-rel-1}
\item[(1)] In $A(\triplet ^{D_2})$ we have
$$ \ell_p ([\omega]) * [H ^{(2)}]= 0, $$
and
$$ \ell_p ([\omega]) * f_p([\omega] ) = 0 $$
where $$ \ell_p(x) = A_p \prod_{i=1} ^{p}(x- h_{4p-i,1}) \prod_{i=1} ^{2p} (x-h_{3p+1/2 -i,1}). $$

\item[(2)]In $\mathcal{P} (\triplet ^{D_2})$ we have
$$ A_p \overline{\omega} ^{3p} \overline{H^{(2)}} = 0$$
and
$$ A_p \overline{\omega} ^{6p-1} = 0. $$
\end{lemma}

\begin{proof}
 We have

 \bea
  0 &=& Q ^{4} ( F ^{(2)}  \tilde{\circ} F ^{(2)} ) =   E^{(2)} \tilde{\circ} F^{(2)} + F^{(2)} \tilde{\circ} E^{(2)}  \nonumber \\
 && +  4 ( Q ^3 e ^{-2 \alpha} \tilde{\circ} Q e ^{-2 \alpha} + Q e ^{-2 \alpha} \tilde{\circ} Q ^3 e ^{-2 \alpha} )  + 6 H ^{(2)} \tilde{\circ} H ^{(2)} \nonumber \\
 &=&  E^{(2)} \tilde{\circ} F^{(2)} + F^{(2)} \tilde{\circ} E^{(2)}  + 6 H ^{(2)} \tilde{\circ} H ^{(2)}  - 4  H ^{(2)}\tilde{\circ} H ^{(2)} \nonumber \\
 && + 4 Q (  ( Q ^2 e ^{-2 \alpha} \circ Q e ^{-2 \alpha} + Q e ^{-2 \alpha} \circ Q ^2 e ^{-2 \alpha} ) \nonumber \\
 &=& E^{(2)}\tilde{\circ} F^{(2)} + F^{(2)} \tilde{\circ} E^{(2)}  + 2 H ^{(2)} \tilde{\circ} H ^{(2)}   \nonumber \\
 && + 4 Q ^2 ( Q e ^{-2 \alpha} \tilde{\circ} Q e ^{-2 \alpha}). \nonumber
 \eea
 Since
 \bea
 && U^{(2)} \tilde{\circ} U^{(2)} = 24 (F ^{(2)} \tilde{\circ} E ^{(2)} + E ^{(2)} \tilde{\circ} F ^{(2)} ), \nonumber \eea
   we get $$ Q ^2 ( Q e ^{-2 \alpha} \tilde{\circ} Q e ^{-2 \alpha}) \in O(\triplet ^{D_2}). $$
On the other hand by Corollary \ref{interp} and Lemma \ref{rel-sing-2}  we have that $$ [Q ^2 ( Q e ^{-2 \alpha} \tilde{\circ} Q e ^{-2 \alpha})]  =  \ell_p([\omega]) * [H ^{(2)}] \in A(\overline{M(1)} ^+). $$
 Now (1) follows from the structure of Zhu's algebra $A(\overline{M(1)}^+)$. Statement (2) follows easily from the proof of  (1) and the fact that
 $$ Q ^2 ( Q e ^{-2 \alpha}   _{-3} Q e ^{-2 \alpha}) \in C_2(\triplet ^{D_2}). $$
\end{proof}

\begin{lemma}\label{lem-d2-rel-11}
We have
$$\ell_p ([\omega]) * [U ^{(2)}]= 0 \quad \mbox{in} \ \  A(\triplet ^{D_2}),$$
$$ A_p \overline{\omega} ^{3p} \overline{U^{(2)}} = 0 \quad \mbox{in} \ \ \mathcal{P} (\triplet ^{D_2}). $$
\end{lemma}

\begin{proof}
Lemma \ref{constr-aut} gives $g \in \mbox{Aut} (\triplet ^{D_2})$ such that
$$ g ( H ^{(2)}) = -1/2 H ^{(2)} - 1/8 U ^{(2)}. $$

 Now assertion follows by applying this automorphism on relations from Lemma \ref{lem-d2-rel-1}.
\end{proof}

\vskip 10mm

By using Corollary \ref{struktura-singlet} and Lemmas \ref{lem-d2-rel-1} and \ref{lem-d2-rel-11} we conclude
$$   \dim A ( \triplet ^{D_2})\le   \dim  \mathcal{P} (\triplet ^{D_2}) \le  12 p-1. $$

Now we shall present a list of $11p$ irreducible modules for $\triplet ^{D_2}$.
We shall get these modules as subquotients of irreducible $\triplet ^{ A_2}$--modules or twisted $\triplet$--modules.

First we recall that in  \cite{ALM}, we have constructed  modules
$$ R(j) = V_{ {\Z} (2 \alpha )+ \frac{3p-j-1/2}{2p} \alpha } , \quad j=1 , \dots, 4p.  $$
These modules are irreducible $\triplet ^{A_2}$--modules and irreducible twisted $\triplet$-modules.
 The following proposition shows that in this way we get $2p$-inequivalent $\triplet ^{D_2}$--modules.

\begin{proposition}We have:
\item[(1)] $R(j)$ are irreducible $\triplet ^{D_2}$--modules.
\item[(2)] $R(j) \cong R( 4p+1-j)$, $j=1, \dots, 2p$.
\item[(3)] Lowest weight of module $R(j)$, with respect to $(L(0), H ^{(2)} (0), U ^{(2) } (0)) $ is
$$ (h_{3p+1/2-j,1}, f_{p} (h_{3p+1/2-j,1}), 0). $$
\end{proposition}
\begin{proof}
First we notice that $\Psi \circ R(j) = R(4p+1-j)$. Therefore $\Psi \circ R(j)$ is not isomorphic to $R(j)$, and   Theorem 6.1 of \cite{DM-Galois} gives that $R(j)$ and $R(4p+1-j)$ are isomorphic as $\triplet ^{D_2}$--modules. This proves assertions (1) and (2). Proof of assertion (3) is easy.
\end{proof}

Let $g \in \mbox{Aut} (\triplet ^{D_2})$ constructed earlier. Then $g$ has order three.

\begin{proposition}
For every $1 \le i \le 2p$, are $R(j) ^g $ and $R(j) ^{g ^{-1}}$ modules irreducible $\triplet^{D_2}$--modules. The lowest weight of $R(j) ^{g}$ and $R(j) ^{g ^{-1}}$ are
$$ (h_{3p+1/2-j,1}, - 1/2 f_p( h_{3p+1/2-j,1}), 6 f_p(h_{3p+1/2-j,1}))$$   and  $$(h_{3p+1/2-j,1}, - 1/2 f_p( h_{3p+1/2-j,1}), -6 f_p(h_{3p+1/2-j,1}))$$
respectively. In this way we have constructed $4p$--irreducible $\triplet^{D_2}$--modules.
\end{proposition}
\begin{proof}
Since $R(j)$ is  an irreducible $\triplet ^{D_2}$--modules, by construction $R(j) ^g $ and $R(j) ^{g ^{-1}}$ must be irreducible. Applying automorphism $g$ we get that lowest weight of $R(j) ^g $
$$ (h_{3p+1/2-j,1}, - 1/2 f_p( h_{3p+1/2-j,1}), 6 f_p(h_{3p+1/2-j,1})).$$
Applying the automorphism $g ^{2} = g ^{-1}$ we get that $R(j) ^{g ^{-1}}$ has lowest weight
$$ (h_{3p+1/2-j,1}, - 1/2 f_p( h_{3p+1/2-j,1}), - 6 f_p(h_{3p+1/2-j,1})).$$
Proof follows.
\end{proof}

\begin{proposition}We have:
\item[(1)] $\Pi(j) ^{\pm}$ are irreducible $\triplet ^{D_2}$--modules.
\item[(2)] $\Pi(j)^+  \cong \Pi(j) ^-$, $j=1, \dots, p$.
\end{proposition}

Consider now a $\Lambda$--family of   $\triplet ^{A_2}$--modules from \cite{ALM}:
$$ \Lambda(i)_0, \Lambda (i)_2, \ i=1, \dots, p.$$
Let $U$ be any irreducible $\triplet ^{A_2}$--module from $\Lambda$--series. Then $U$ is $\Psi$--invariant and
 therefore splits into a  direct sum of two irreducible $\triplet ^{D_2}$--modules:
$$ U = U ^+ \oplus U  ^-; \quad U ^{\pm}= \{ v \in U \ \vert \  \Psi(u) = \pm u \}. $$
 In this way we get $4p$ non-isomorphic $\triplet^{D_2}$--modules.

\begin{proposition}
Modules $ \Lambda(i)_0 ^{ \pm}$ and $\Lambda(i)_ 2 ^{\pm}$ are non-isomorphic irreducible $\triplet ^{D_2}$--modules. We have:
$$ (\Lambda(i)_0 ^- ) ^g = \Lambda(i) _2 ^-, \ \  (\Lambda(i)_0 ^- ) ^{g^{-1}} = \Lambda(i) _2 ^+. $$
\end{proposition}
\begin{proof}
By construction and using Galois theory we have that  these modules are irreducible and non-isomorphic. Let us identify lowest weights.
Clearly, lowest weight of $\Lambda(i)_0 ^+$ is $(h_{i,1}, 0, 0)$.

Lowest weight of $\Lambda(i) ^-$ has the form $ (h_{4p-i,1}, -2 f_p (h_{4p-i,1}), 0 )$. Applying the automorphism $g$ we get that lowest weights of  $\Lambda(i)_2 ^{\pm}$ are
$$ (h_{4p-i,1}, f_{p} (h_{4p-i,1}), \pm 12 f_{p} (h_{4p-i,1}) ). $$
\end{proof}
In this way we have constructed  $ 11 p$ non-isomorphic irreducible modules. Since $\triplet$ has $p-1$ logarithmic representations obtained in \cite{am2}, we get that
$$\dim A( \triplet ^{D_2}) \ge 12 p-1. $$

\begin{lemma} \label{com-2}
In Zhu's algebra  $A( \triplet ^{D_2})$ we have:
 $$[H^{(2)}] * [U ^{(2)} ] = [U ^{(2)}] *  [H^{(2)}] = g_p([\omega]) * [U ^{(2)}],$$
 where $g_p$ is a polynomial of degree $\deg  {g_p} \le 3p-1$.
 In particular, $A( \triplet ^{D_2})$ is a commutative algebra.
\end{lemma}
\begin{proof}
By definition of multiplication in Zhu's algebra we see that
 $$[H^{(2)}] * [U ^{(2)} ] =\widetilde{r_1} ([\omega]) * [U  ^{(2)}], \ [U ^{(2)}] *  [H^{(2)}] = \widetilde{r_2}([\omega]) * [U ^{(2)}],$$
 for certain polynomials $\widetilde{r_1}, \widetilde{r_2}$  of degree $\deg \widetilde{r_1} \le 3p-1$, $\deg \widetilde{r_2} \le 3p-1 $. By applying the above relations on lowest component of $\triplet ^{D_2}$--modules we get:
 $$ \widetilde{r_1}  = \widetilde{r_2}  = g_p $$
 where $g_p$ is a polynomial of degree $\deg  {g_p} \le 3p-1$ satisfying the following interpolation conditions:
 $$ g_p(h_{4p-i,1} ) =  f_p (h_{4p-i,1} ), \qquad g_{p} (h_{3p+1/2-j,1}) = - \tfrac{1}{2} f_p (h_{3p+1/2-j,1}), $$
for $i=1, \dots, p$, $j=1, \dots, 2p$.
 The proof follows.
\end{proof}

We have:

\begin{theorem}
Assume that $A_p \ne 0$. Then
$$ \dim  \mathcal{P} (\triplet ^{D_2}) = \dim A ( \triplet ^{D_2}) = 12 p-1. $$
Moreover, the set
$$ \{ \Lambda(i)_0 ^{\pm} ,  \Lambda(i)_2 ^{\pm} ,   \Pi(i)^+,   R(j),  {R(j)} ^{g}, {R(j)} ^{ g ^{-1}} , i=1, \dots,p, \ j=1, \dots, 2p.  \}$$
provides a complete list of  irreducible $\triplet ^{D_2}$--modules.
\end{theorem}

\begin{corollary}
Zhu's algebra $ A( \triplet ^{D_2})$ is a $12 p-1$ dimensional commutative algebra.
\end{corollary}

\begin{proof}
We will use our result on classification of irreducible modules and dimension of Zhu's algebra. Zhu's algebra $A(\triplet ^{D_2})$ is isomorphic to a direct product of $10 p+ 1 $ $1$-dimensional ideals ideals, and $(p-1)$ $2$-dimensional indecomposable ideals which corresponds to logarithmic modules constructed in \cite{am2}.
\end{proof}

\begin{remark}
By comparing lowest weights (and using Zhu's algebra theory), one gets that
$$ R(j) ^{ h g} \cong R(j) ^{g ^{-1}}$$
as $\triplet ^{D_2}$--modules.
\end{remark}

\begin{corollary}
The vector space spanned by   irreducible characters of $\triplet ^{D_2}$-modules is $ 5 p$ dimensional with basis
$$ \{ \mbox{ch} \  \Lambda(i)_0 ^{\pm},\mbox{ch} \   \Pi(i) ^+,   \mbox{ch} \  R(j), \ i=1, \dots, p, \ j=1, \dots, 2p \}, $$
and it is the same as the vector space spanned by irreducible characters for $\triplet ^{A_2}$.
Moreover, the modular closure of irreducible characters of $\triplet ^{D_2}$-modules is $6 p-1$ dimensional and it is the same as that for $\triplet ^{A_2}$.
\end{corollary}

Let us list explicitly lowest weights of irreducible $\triplet ^{D_2}$--modules.
\vskip 3mm
\begin{center}
  \begin{tabular}{|c|c|c|c|}
    \hline
     \mbox{module}  $M$ & $L(0)$ & $ H^{(2)}(0) $ &  $ U^{(2)}(0) $\\ \hline \hskip 2mm
   $\Lambda(i)^+ _0 $ & $  h_{i,1}$& $ 0  $ & $0 $ \\ \hline
      $\Lambda(i)^- _0 $ & $  h_{i,3}$& $ -2f_p (h_{i,3}) $ & $0 $ \\ \hline
   $ \Lambda(i)^+_2  $  & $ h_{i, 3} $ & $  f_p ( h_{i, 3}) $ & $ 12 f_p ( h_{i, 3} )  $ \\
   \hline
    $ \Lambda(i)^-_2  $  & $ h_{i, 3} $ & $  f_p (h_{i, 3}) $ & $-12 f_p ( h_{i, 3} ) $ \\
   \hline
 $ \Pi(i)  ^+ \cong\Pi(i)  ^- $ & $ h_{p+i,3 } $ & $  f_p ( h_{p+i,3 } ) $  & $0$ \\ \hline
    $R(j)$  &  $ h_{3p + 1/2- j,1} $&$ f_p (h_{3p+1/2-j,1})$&   $ 0  $ \\
    \hline
        $R(j)^{g }$  &  $ h_{3p + 1/2- j,1} $&$- \frac{1}{2} f_p (h_{3p+1/2-j,1})$& $ 6   f_p (h_{3p+1/2-j,1}) $ \\
    \hline
        $R(j)^{ g ^{-1} }$  &  $ h_{3p + 1/2-j,1}$&$ - \frac{1}{2} f_p ( h_{3p+1/2-j,1}) $& $- 6  f_p (h_{3p+1/2-j,1}) $ \\
    \hline
  \end{tabular}
\end{center}

\section{$C_2$--cofiniteness of the vertex algebra $\triplet ^{D_m}$ }
\label{general-1}
In previous section, we described Zhu's algebra and $C_2$--algebra for the vertex algebra $\triplet ^{D_m}$ in the case $m=2$. For these results we needed a technical condition that the polynomial $\ell_p(x)$ is a non-trivial polynomial.
But non-triviality of polynomial  $\ell_p(x)$ is related to certain constant term identity for which we do not have proof in full-generality (We checked this for $p$ small).

Now we shall consider $\triplet ^{D_m}$. We know that this vertex
algebra is strongly generated by $\omega$, $H ^{(2)}$ and $ U
^{(m)}=(2m)! F ^{(m)} + E ^{(m)}$.  As in \cite{ALM}, we have
natural homomorphism
 $$\Phi: \mathcal{P}({\overline{M(1)} } ^+) \rightarrow \mathcal{P}( M(1)), \ \
  \overline{\omega} \mapsto \frac{1}{4p} \beta ^2 , \ \overline{H ^{(2)} } \mapsto  \frac{a}{  (4p) ^{3p-1} } \beta ^{ 6p-2} ,$$
  where $a \ne 0$.

\begin{lemma} \label{non-triv-gen-2}

Assume that $m \ge 2$. Then there exists an odd integer $n$, $2< n < 4m(p-1)+4$
such that
 $$ U ^{(m)} _{-n} U ^{(m)}  \in \overline{M(1)}^+ \quad \mbox{and} \quad  \Phi (\overline{ U ^{(m)} _{-n} U ^{(m)} } ) \ne 0. $$
\end{lemma}

\begin{proof}
Assume the contrary. Then we have, for each odd integer $n$, $2< n< 2m(p-1)+4$,

{\setlength{\arraycolsep}{0pt}
\begin{eqnarray*}
&&\overline{U ^{(m)} _{-n} U ^{(m)}}\\
&=&\overline{(2m)!(E ^{(m)} _{-n} F ^{(m)}+F ^{(m)} _{-n} E ^{(m)})} \\
&=&\overline{2(2m)!(F ^{(m)} _{-n} E ^{(m)})}\\
&=&0,
\end{eqnarray*}
} in $M(1)/C_{2}(M(1))$.

 By the above computation and the proof of
Lemma 3.3 of \cite{ALM}, we obtain the following identity

{\setlength{\arraycolsep}{0pt}
\begin{eqnarray*}
&&e^{-mx_{0}}x_{0}^{-2m^{2}p}(b_{0}+\cdots+b_{2m(-p+1)}x_{0}^{2m(-p+1)})\\
&=&\sum_{n=-2m(mp+p-1)}^{+\infty}c_{n}x_{0}^{n},
\end{eqnarray*}
} where $b_{2m(-p+1)}\neq0$, and $c_{n}=0$ when $n$ is an odd integer and  $2< n <
4m(p-1)+4$. This contradicts Lemma \ref{apa} below and
completes the proof.
\end{proof}

  The following lemma will give a non-trivial polynomial of possibly larger degree which annihilates $([H ^{(2)}] - f_p([\omega]))$.
\begin{lemma} \label{non-triv-gen-1}
\item[(1)] In $A (\overline{M(1)} ^+)$ we have
$$ \ell([\omega]) * ([H ^{(2)}] - f_p([\omega]) ) = 0$$
for certain polynomial $\ell \in {\C}[x]$, $\deg \ell \ge 3p$, $\ell_p(x)  \mid \ell(x)$.
\item[(1')] In $A (\triplet ^{D_2})$ we have
$$ \ell([\omega]) * ([H ^{(2)}]) = 0$$
for certain polynomial $\ell \in {\C}[x]$, $\deg \ell \ge 3p$, $\ell_p(x)  \mid \ell(x)$.

\item[(2)] In $\mathcal{P}(\overline{M(1)}^+)$ we have
$$ \overline{\omega} ^n ( \overline{H ^{(2)} } - a \overline{\omega} ^{3p-1})  = 0$$
where $a \ne 0$ and $n \ge 3p$.

\item[(2')] In $\mathcal{P}( \triplet ^{D_2})$ we have
$$ \overline{\omega} ^n ( \overline{H ^{(2)} }) = 0, \quad    \overline{\omega} ^{3p + n-1}  = 0$$
where   $n \ge 3p$.

\end{lemma}

\begin{proof}
Take $ k \ge 2$ such that
$$ \Phi( \overline{ U ^{(2)} _{-k} U ^{(2)} } ) \ne 0. $$
As before, in $A(\overline{M(1)} ^+)$ we have relation
$$ [H ^{(2)} \tilde{\circ}_k  H ^{(2)}] = \ell ([\omega]) * ([H ^{(2)}] - f_p([\omega])), $$
for certain polynomial $\ell \in {\C}[x]$, and
$$ \ell ([\omega]) * [H ^{(2)}] = \nu  [Q ^2 ( Q e ^{-2 \alpha} \tilde{\circ}_k Q e ^{-2 \alpha})] \quad (\nu \ne 0). $$

As in the proof of Lemma \ref{lem-d2-rel-1} we have
 $$ E^{(2)}\tilde{\circ}_k F^{(2)} + F^{(2)} \tilde{\circ}_k E^{(2)}  + 2 H ^{(2)} \tilde{\circ}_k  H ^{(2)}
  + 4 Q ^2 ( Q e ^{-2 \alpha} \tilde{\circ}_k Q e ^{-2 \alpha}) = 0. $$

This implies that in $A(\overline{M(1)})$:
$$  [ Q ^2 ( Q e ^{-2 \alpha} \tilde{\circ}_k Q e ^{-2 \alpha})] = \widetilde{\ell}([\omega]) [H ^{(2)}] $$
for certain non-trivial polynomial $\widetilde{\ell} (x)$. This proves the assertion (1). Assertion (1') follows from the fact that $$ Q ^2 ( Q e ^{-2 \alpha} \tilde{\circ}_k Q e ^{-2 \alpha}) \in O(\triplet ^{D_2}). $$
Relations (2) and (2') follows from (1) and (1').
\end{proof}

\begin{remark}
Lemma \ref{non-triv-gen-1}  gives another proof of $C_2$--cofiniteness of $\triplet ^{D_2}$ without using Conjecture \ref{slutnja-1}.
\end{remark}

Lemma \ref{non-triv-gen-1}
 easily implies  $\mbox{Ker} (\Phi)$ is an nilpotent ideal in  $\mathcal{P}({\overline{M(1)} }^+)$.

Now we are in the position to prove $C_2$--cofiniteness of $ \triplet ^{D_m}$. As in the case $m=2$ we have
$$ \mathcal{P} (\triplet ^{D_m} ) = \mathcal{P}_0 (\triplet ^{D_m}) \oplus \mathcal{P}_m (\triplet ^{D_m})$$
where
$\mathcal{P}_0 (\triplet ^{D_m})$ is a subalgebra of $\mathcal{P} (\triplet ^{D_m})$ generated by $\overline{\omega}$ and $\overline{H ^{(2)}}$ and
$$ \mathcal{P}_m (\triplet ^{D_m}) = \mathcal{P}_0 (\triplet ^{D_m}). \overline{U ^{(m)}}. $$

\begin{theorem}
The vertex operator algebra $\triplet ^{D_m}$ is $C_2$--cofinite.
\end{theorem}
\begin{proof} $C_2$--cofiniteness of $\triplet ^+$ was proved in Proposition \ref{m1}. We shall here present the proof for the case   $m \ge 2$. It suffices  to prove that ${\mathcal{P}_0 (\triplet ^{D_m})}$ is finite--dimensional.
Lemma \ref{non-triv-gen-1}  gives relation
$$ \overline{\omega} ^n ( \overline{H ^{(2)} } - a \overline{\omega} ^{3p-1}) = 0 \ \mbox{in} \ {\mathcal{P}_0 (\triplet ^{D_m})}, $$
Now Lemma \ref{non-triv-gen-2}  shows existence of the  non-trivial element
$$  U ^{(m)}_{-k} U ^{(m)} \in \mathcal{P} (\overline{M(1)} ^+ ) \cap C_2 (\triplet ^{D_m}), $$
which gives relation
   $$\overline{ U ^{(m)}_{-k} U ^{(m)}  } =  A(\overline{\omega}) ( \overline{H ^{(2)} } - a \overline{\omega} ^{3p-1})  + B(\overline{\omega})= 0 \ \mbox{in} \ {\mathcal{P}_0 (\triplet ^{D_m})}, $$
   for certain polynomials $A, B \in {\C}[x]$. Since $\Phi  (U ^{(m)}_{-k} U ^{(m)}) \ne 0$ we have that $ B \ne 0$.   By multiplying previous  relation with $\overline{\omega} ^n$ we get
   $$ \overline{\omega} ^n B(\overline{\omega}) = 0 \ \mbox{in} \ {\mathcal{P}_0 (\triplet ^{D_m})}. $$
   This easily proves that ${\mathcal{P}_0 (\triplet ^{D_m})}$ is finite-dimensional.
\end{proof}

We still have an unproved lemma.

 \begin{lemma}\label{apa}
Let $p$ be a non-negative integer and let $b_{0}+\cdots+b_{p}x^{p}$ be a polynomial in $\mathbb{C}[x]$,
$b_{0}\neq0$ and let $m$ be a nonzero constant. Then in the
following product expansion {\setlength{\arraycolsep}{0pt}
\begin{eqnarray*}
&&e^{mx}(b_{0}+\cdots+b_{p}x^{p})\\
&=&\sum_{n=0}^{\infty}c_{n}x_{0}^{n}\in\mathbb{C}[[x]],
\end{eqnarray*}
} there cannot exist an
$n,k\in\mathbb{Z}_{>0}$ such that
$c_{n+p}=c_{n+p+k}=\cdots=c_{n+p+pk}=0$.
\end{lemma}
\begin{proof}Without loss of generality, we can set $m=1$.
Assume the lemma is false, then there exist an
$n\in\mathbb{Z}_{>0}$ such that
$c_{n+p}=c_{n+p+k}=\cdots=c_{n+p+pk}=0$, i.e.,
\[\begin{matrix}
\begin{pmatrix}\frac{1}{n!}&\frac{1}{(n+1)!}& \cdots &\frac{1}{(n+p)!}\\
\frac{1}{(n+k)!}&\frac{1}{(n+k+1)!}& \cdots &\frac{1}{(n+k+p)!}\\
 \vdots &  \vdots &  &\vdots\\\frac{1}{(n+pk)!}&\frac{1}{(n+pk+1)!}&\cdots&\frac{1}{(n+pk+p)!} \end{pmatrix}
\begin{pmatrix}b_{p}\\ \vdots \\b_{0}\end{pmatrix}=0
\end{matrix}\]

But direct computation (by using column operations and induction on $p$) shows that the determinant
$$\begin{vmatrix}\frac{1}{n!}&\frac{1}{(n+1)!}& \cdots &\frac{1}{(n+p)!}\\
\frac{1}{(n+k)!}&\frac{1}{(n+k+1)!}& \cdots &\frac{1}{(n+k+p)!}\\
 \vdots &  \vdots &  &\vdots\\\frac{1}{(n+pk)!}&\frac{1}{(n+pk+1)!}&\cdots&\frac{1}{(n+pk+p)!} \end{vmatrix}=(-1)^{p(p+1)/2}\prod^{p+1}_{i=1}\frac{(j-1)!k^{i-1}}{(n+(i-1)k+p)!},$$
a contradiction. This completes the proof.
\end{proof}

\section{Towards irreducible $\triplet ^{D_m}$-modules}
\label{general-2}
In the section, we assume that the readers are familiar with the enumeration of irreducible modules in Section 4 of \cite{ALM}.

We shall first assume that $ m= 2k$ is even.
Recall the conjectural set of irreducible $\am$--modules.

First we have  $\Lambda$-series of irreducible modules of $\am$:

$$\Lambda(i)_0, \Lambda(i)_j ^{\pm}, \Lambda(i)_m, \ (j= 1, \dots, k-1). $$

$ \Lambda(i)_0 $ and $ \Lambda(i)_m$ are $\Psi$-invariant, and therefore:
 we have decomposition of $\triplet ^{D_m}$-modules:
$$\Lambda(i)_0\cong\Lambda(i)^+_0\bigoplus\Lambda(i)^-_0,$$
$$\Lambda(i)_m\cong\Lambda(i)^+_m\bigoplus\Lambda(i)^-_m,$$
where $\Lambda(i)^{\pm}_0, \Lambda(i)^{\pm}_m$ are the eigenspaces  for eigenvalues $\pm1 $. By quantum Galois theory, $\Lambda(i)^{\pm}_0, \Lambda(i)^{\pm}_m$ are non-isomorphic
irreducible $\triplet ^{D_m}$--modules.

Since $\Psi( \Lambda(i)_j ^{\pm} ) = \Lambda (i)_j ^{\mp}$ we conclude that $\Lambda(i)_j ^{\pm}$ are irreducible $\triplet ^{D_m}$--modules and
$$ \Lambda(i)_j ^+ \cong \Lambda (i) _j ^{-} .$$

So  $\Lambda$-series of irreducible $\triplet ^{A_m}$--modules, gives $4 p + (m /2  -1) p$--irreducible $\triplet ^{D_m}$--modules:

 $$ \Lambda(i)_0 ^{\pm}, \Lambda(i)_m ^{\pm}, \Lambda(i) _j ^+ \quad (j=1, \dots, k). $$

We also have $\Pi$-series of irreducible $\triplet ^{A_m}$--modules:
$$ \Pi(i)_j ^{\pm}, \ (j=1, \dots, k ). $$
Since $\Psi( \Pi(i)_j ^{\pm} ) \cong \Pi(i)_j ^{\mp}$, we conclude that
$\Pi(i)_j ^{\pm}$ are irreducible $\triplet ^{D_m}$--modules and
$$\Pi(i)^-_j\cong\Pi(i)^+_j.$$

 So $\Pi$ series of irreducible $\triplet ^{A_m}$--modules, gives $m p /2 $--irreducible non-isomorphic $\triplet ^{D_m}$--modules:
 $$\Pi(i)_j ^{+}  \quad (j=1, \dots, k).$$

In this way we have proved that $\Lambda$ and $\Pi$--series gives $( 3 +m )p$--irreducible non-isomorphic  $\triplet ^{D_m}$-modules. By using similar analysis we get the same result in the case $m= 2k+1$.

\begin{lemma}
The $\Lambda$ and $\Pi$-series of irreducible $\triplet ^{A_m}$--modules gives  $(3+m) p$ non-isomorphic irreducible $\triplet ^{D_m}$--modules.
\end{lemma}

 We also have the twisted series of irreducible  $\triplet ^{A_m}$--modules:

For $k=0,...,m-1$, we let
$$R(i,j,k):=\bigoplus_{s \in \mathbb{Z}} M(1) \otimes e^{\frac{j-\frac{i}{m}}{2p} \alpha+(m s+k) \alpha}$$
Since
$$ \Psi (R(i,j,k) ) \cong R(m-i, 2p-j-1, m-k),  $$
we have that each $R(i,j,k)$ is an irreducible $\triplet ^{D_m}$-module and
$R(i,j,k)\cong R(m-i,2p-j-1,m-k)$
In this way we get   $pm(m-1)$ irreducible $\triplet ^{D_m}$-modules.

We also have $4p$ irreducible modules
$$ R (j) ^{\sigma}, R(j) ^{h \sigma}$$
constructed in Proposition \ref{r-sigma-modules}. We call these modules as $R ^{\sigma}$--series.

 Finally, we conclude:

\begin{proposition} \label{konstruirani-moduli}
The $\Lambda$, $\Pi$,  $R$ and $R ^{\sigma}$--series give $(m ^2 + 7) p$ non-isomorphic irreducible $\triplet ^{D_m}$--modules.
\end{proposition}

Realization of irreducible modules enables us to prove commutativity of Zhu's algebra.
\begin{proposition}
Let $m \ge 3$. In Zhu's algebra $A(\triplet ^{D_m} )$ we have
$$ [H ^{(2)} ] * [U ^{(m)}] =  [U ^{(m)}] * [H ^{(2)} ] =g_p ^m ([\omega]) * [U ^{(m)}]$$
where $g_p ^{m}$ is a polynomial of degree at most $3p-1$ which satisfies
$$ g_p ^{m} (h_{i, m+1}) =  f_p( h_{i,m+1}), \  g_p ^{m} (h_{3p+1/2-j , 1}) = -1/2 f_p ( h_{3p+1/2-j , 1} )   \quad ( m \ \mbox{even}) ,  $$
$$ g_p ^{m} (h_{p+i, m+2}) =  f_p( h_{p+i,m+2}), \  g_p ^{m} (h_{3p+1/2-j , 1}) = -1/2 f_p ( h_{3p+1/2-j , 1} )   \quad ( m \ \mbox{odd}).  $$
In particular, Zhu's algebra $A( \triplet ^{D_m})$ is commutative.
\end{proposition}
\begin{proof}
The proof is similar to that of Lemma \ref{com-2}. We evaluate relations
$$ [H ^{(2)} ] * [U ^{(m)}] = \widetilde{r_1} ([\omega]) * [U ^{(m)}], \ [U ^{(m)} ] * [H ^{(2)}] = \widetilde{r_2} ([\omega]) * [U ^{(m)}], $$
on lowest components of $\triplet ^{D_m}$--modules
$$ R(j) ^{\sigma}, \Lambda(i) _m ^{\pm}, \ \ j=1, \dots, 2p, \ i= 1, \dots, p \quad (m \ \mbox{even}), $$
$$ R(j) ^{\sigma}, \Pi(i) _m ^{\pm}, \ \ j=1, \dots, 2p, \ i= 1, \dots, p \quad (m \ \mbox{odd}). $$
\end{proof}

The above analysis, and the results from \cite{ALM}  and from Section \ref{section-d2} in the cases $m=1,2$  motivate the following conjecture:

\begin{conjecture} \label{conj-general}
\item[(1)] The vertex algebra  $\triplet ^{D_m}$ has $(m ^2 + 7) p$ non-isomorphic irreducible modules, and  the modules constructed in Proposition \ref{konstruirani-moduli}
give a complete list of irreducible $\triplet ^{D_m}$--modules.

\item[(2)] If $m \ge 2$,  then  Zhu's algebra $A(\triplet ^{D_m} )$ is commutative algebra of dimension  $(m ^2 + 8 )p-1$.

\end{conjecture}

\subsection{Irreducible $\triplet ^{D_m}$-modules: lowest weights}

Since we have already proved that modules in Proposition \ref{konstruirani-moduli}  are non-isomorphic, we don't need explicit formulas for lowest weights. But for completeness, we shall list explicit formulas for lowest weights without proof for general $m$ (we calculated these weights in the case $m=2$).

The action of Zhu's algebra on the lowest weights spaces $M(0)$ can be found in the following two tables,
where
%
$$\phi(t)=(-1)^{p}\prod_{l=0}^{m-1}{t+pl \choose (m+1)p-1}\frac{((m+1)p-1)!((l+1)p)!}{((m+l+1)p-1)!p!},$$
 and number $\ell$ is defined as in Lemma 4.8 of \cite{ALM}.

\vskip 5mm

Let $m = 2k$.
\vskip 3mm
\begin{center}
  \begin{tabular}{|c|c|c|c|}
    \hline
     \mbox{module}  $M$ & $L(0)$ & $ H^{(2)}(0) $ &  $ U^{(m)}(0) $\\ \hline \hskip 2mm
   $\Lambda(i)^+ _0 $ & $  h_{i,1}$& $ 0  $ & $0 $ \\ \hline
      $\Lambda(i)^- _0 $ & $  h_{i,3}$& $ -2f_p (h_{i,3}) $ & $0 $ \\ \hline
  $ \Lambda(i)_j ^ +\cong\Lambda(i)_j ^ -  $  & $  h_{i, 2 j +1} $& $ f_p ( h_{i, 2 j +1}  )  $ & $0 $ \\
   \hline
   $ \Lambda(i)^+_m  $  & $ h_{i, m +1} $ & $ f_p ( h_{i, m+1}) $ & $\frac{(2m)!}{m!}\phi(-mp+i-1) $ \\
   \hline
    $ \Lambda(i)^-_m  $  & $ h_{i, m +1} $ & $ f_p (h_{i, m+1}) $ & $-\frac{(2m)!}{m!}\phi(-mp+i-1) $ \\
   \hline
 $ \Pi(i)_j  ^+ \cong\Pi(i)_j  ^- $ & $ h_{p+i,2 j+1} $ & $  f_p ( h_{p+i,2 j+1} ) $  & $0$ \\ \hline
    $R( i, j,k)$  &  $h_{\ell + 1 - i/m,1}$&$f _p (  h_{\ell + 1 - i/m,1} )$& $0$ \\
    \hline
       $R(j)^{\sigma}$  &  $ h_{3p + 1/2- j,1} $&$- \frac{1}{2} f_p (h_{3p+1/2-j,1})$& $\frac{(2m)!}{2^{m-1}m!}\phi(3p - 1/2- j)$ \\
    \hline
        $R(j)^{ h \sigma }$  &  $ h_{3p + 1/2-j,1}$&$ - \frac{1}{2} f_p ( h_{3p+1/2-j,1}) $& $-\frac{(2m)!}{2^{m-1}m!}\phi(3p - 1/2- j)$ \\
    \hline
  \end{tabular}
\end{center}

\vskip 3mm

\vskip 5mm

Let $m = 2k+1$.
\vskip 3mm
\begin{center}
  \begin{tabular}{|c|c|c|c|}
    \hline
     \mbox{module}  $M$ & $L(0)$ & $ H^{(2)}(0) $ &  $ U^{(m)}(0) $\\ \hline \hskip 2mm
   $\Lambda(i)^+ _0 $ & $  h_{i,1}$& $ 0  $ & $0 $ \\ \hline
      $\Lambda(i)^- _0 $ & $  h_{i,3}$& $ -2 f_p (h_{i,3}) $ & $0 $ \\ \hline
  $ \Lambda(i)_j ^ +\cong\Lambda(i)_j ^ -  $  & $  h_{i, 2 j +1} $& $  f_p (h_{i, 2 j +1} ) $ & $0 $ \\
   \hline
 $ \Pi(i)_j  ^+ \cong\Pi(i)_j  ^- $ & $ h_{p+i,2 j+1} $ & $ f_p ( h_{p+i,2 j+1} ) $  & $0$ \\ \hline
    $ \Pi(i)^+_m  $  & $ h_{p+i, m+2} $ & $ f_p  ( h_{p+i, m+2} )   $ & $\frac{(2m)!}{2^{m-1}m!}\phi(-mp+i-1) $ \\
   \hline
    $ \Pi(i)^-_m  $  & $ h_{p+i, m+2} $ & $ f_p (  h_{p+i, m+2} )$ & $-\frac{(2m)!}{2^{m-1}m!}\phi(-mp+i-1)$ \\
   \hline
    $R( i, j,k)$  &  $h_{\ell + 1 - i/m,1}$&$ f_p  ( h_{\ell + 1 - i/m,1} )$& $0$ \\
\hline
   $R(j)^{\sigma}$  &  $ h_{3p + 1/2- j,1} $&$- \frac{1}{2} f_p (h_{3p+1/2-j,1})$& $\frac{(2m)!}{2^{m-1}m!}\phi(3p - 1/2- j)$ \\
    \hline
        $R(j)^{ h \sigma }$  &  $ h_{3p + 1/2-j,1}$&$ - \frac{1}{2} f_p ( h_{3p+1/2-j,1}) $& $-\frac{(2m)!}{2^{m-1}m!}\phi(3p - 1/2- j)$ \\
    \hline
  \end{tabular}
\end{center}

\section{Irreducible characters and modular closure}

As in \cite{ALM} we compute the $SL(2,\mathbb{Z})$-closure of the
character of $\triplet ^{D_m}$. We shall relate irreducible characters of $\triplet ^{D_m}$ with irreducible characters of $\triplet ^{A_m}$ constructed from \cite{ALM}.

First we notice that by construction ${\rm ch} R(j) ^{\sigma}= {\rm ch} R(j) ^{h \sigma}$ are characters of irreducible $\triplet ^{A_2}$--modules.
Therefore, if $m$ is even,  these characters are in the span of characters of irreducible $\triplet ^{A_m}$--modules. But if $m$ is odd, we get $2p$ new characters which are linearly independent with $\triplet ^{A_m}$--characters

The case $m =1$ was studied in \cite{ALM} since $\triplet^{D_1} \cong \triplet ^{A_2}$. The case $m=2$ was also studied in previous sections.
As usual for a $V$-module $M$ we denote
$${\rm ch}_{M}:={\rm tr}_M q^{L(0)-c/24}.$$

\begin{lemma}
For every $m \in {\N}$ and $1 \le i \le p$
$$  \mbox{ch}_ {L_{ \overline{M(1)} } (h_{i,1},0) ^+}   - \mbox{ch}_{L_{ \overline{M(1)} } (h_{i,1},0) ^-}    $$
is in the linear span of characters of irreducible $\triplet ^{A_2}$--modules. In particular,  $\mbox{ch} \Lambda(i)_0 ^{\pm}$ is in linear span of irreducible $\triplet ^{A_m}$ and irreducible $\triplet ^{A_2}$--modules.
\end{lemma}
\begin{proof}
As an $\overline{M(1)} ^+$--modules we have
$$ \Lambda(i)^+_0 = \overline{M(1)} ^+ . e ^{\frac{i-1}{2p} \alpha} \bigoplus R_0 ^+, \quad  \Lambda(i)^-_0 = \overline{M(1)} ^+ . Qe ^{\frac{i-1}{2p} \alpha - \alpha} \bigoplus R_0 ^- , $$
where $R_0 ^{+} $ and $R_0 ^{-}$ are isomorphic $\overline{M(1)} ^+$--modules.
The above decomposition follows for every $m$ and therefore
$$ \mbox{ch}_{ \Lambda(i)_0 ^+} - \mbox{ch}_{ \Lambda(i) _0 ^-} =   \mbox{ch}_{\overline{M(1)} ^+ . e ^{\frac{i-1}{2p} \alpha}}  - \mbox{ch}_{ \overline{M(1)} ^+ . Qe ^{\frac{i-1}{2p} \alpha - \alpha}} $$
is in linear span of $\triplet ^{D_2}$--modules, which is the same as linear span of irreducible $\triplet ^{A_2}$--modules.
\end{proof}
As Virasoro modules, we have

$$\Lambda(i)^+_0=(\bigoplus_{n=1}^\infty (n) \bigoplus_{k=0}^{m-1} L(c_{p,1},h_{i,2(nm+k)+1}))\bigoplus(\bigoplus_{k=0}^{\infty} L(c_{p,1},h_{i,4k+
1})).$$

$$\Lambda(i)^-_0=(\bigoplus_{n=1}^\infty (n) \bigoplus_{k=0}^{m-1} L(c_{p,1},h_{i,2(nm+k)+1}))\bigoplus(\bigoplus_{k=0}^{\infty} L(c_{p,1},h_{i,4k+
3})).$$

Hence, {\setlength{\arraycolsep}{0pt}
\begin{eqnarray*}
&&{\rm ch}_{\Lambda(i)^+_0}(\tau)\\
&=&\frac{1}{\eta(\tau)} (\sum_{n \geq 1}^\infty n q^{p(mn+\frac{p-i}{2p})^2}-\sum_{n \geq 2}^\infty (n-1) q^{p(mn-(m-1)-\frac{p-i}{2p})^2})  \\
&+&\frac{1}{\eta(\tau)} (\sum_{n \geq 1}^\infty n q^{p(mn+1+\frac{p-i}{2p})^2}-\sum_{n \geq 2}^\infty (n-1) q^{p(mn-(m-2)-\frac{p-i}{2p})^2})  \\
&+& \cdots\cdots\cdots\\
&+&\frac{1}{\eta(\tau)} (\sum_{n \geq 1}^\infty n q^{p(mn+m-1+\frac{p-i}{2p})^2}-\sum_{n \geq 2}^\infty (n-1) q^{p(mn-\frac{p-i}{2p})^2})  \\
&+&\frac{1}{\eta(\tau)} \sum_{n  \geq 0}^\infty ( q^{p(2n+\frac{p-i}{2p})^2}- q^{p(2n+\frac{p+i}{2p})^2})\\
&=&\frac{1}{\eta(\tau)} \sum_{n\in\mathbb{Z}}n ( q^{p(mn+\frac{p-i}{2p})^2}+ q^{p(mn+1+\frac{p-i}{2p})^2}+\cdots+ q^{p(mn+m-1+\frac{p-i}{2p})^2})  \\
&+&\frac{1}{\eta(\tau)} \sum_{n  \geq 1}^\infty q^{p(n-\frac{p-i}{2p})^2}+\frac{1}{\eta(\tau)} \sum_{n  \geq 0}^\infty ( q^{p(2n+\frac{p-i}{2p})^2}- q^{p(2n+\frac{p+i}{2p})^2})\\
&=&\frac{1}{2}(Q_{m(p-i),pm^2}(\tau)+Q_{m(3p-i),pm^2}(\tau)+\cdots + Q_{m(2pm-p-i),pm^2}(\tau))+\Theta_{2(p-i),4p}(\tau),\\
\end{eqnarray*}
}
where $Q_{\cdot,\cdot}(\tau)$ are defined in \cite{ALM}.
Similarly we have

{\setlength{\arraycolsep}{0pt}
\begin{eqnarray*}
&&{\rm ch}_{\Lambda(i)^-_0}(\tau)\\
&=&\frac{1}{2}(Q_{m(p-i),pm^2}(\tau)+Q_{m(3p-i),pm^2}(\tau)+\cdots + Q_{m(2pm-p-i),pm^2}(\tau))+\Theta_{2(p+i),4p}(\tau).\\
\end{eqnarray*}
}


We also note that
 $$ {\rm ch}_{\Lambda(i)^{\pm}_m}=\frac{1}{2}{\rm ch}_{\Lambda(i)_m}$$
when $m$ is even, and
 $$ {\rm ch}_{\Pi(i)^{\pm}_m}=\frac{1}{2}{\rm ch}_{\Pi(i)_m}$$
when $m$ is odd.

Now by the calculation of characters and Theorem 6.5 in \cite{ALM} we summarize:

\begin{theorem} \label{modular}

\item[(1)] When $m$ is even, the dimension of the modular closure of irreducible modules $\Lambda$, $\Pi$,  $R$  and $R ^{\sigma}$ families (conjecturally all irreducible modules of $\triplet ^{D_m}$) is $ (m ^2 + 2) p -1$. Moreover, the space coincide with that of $\triplet ^{A_m}$.

\item[(2)]When $m$ is odd and $m \ge 3$,  the modular closure of irreducible modules $\Lambda$, $\Pi$, $R$   and $R ^{\sigma}$ families (conjecturally all irreducible modules of $\triplet ^{D_m}$)
is $(m^2+ 5 )p -1$-dimensional and it is spanned by the  $ (m ^2 + 2 ) p -1$--dimensional space coming from $\triplet ^{A_m}$ together
with
$3p $ irreducible characters
$$ {\rm ch}_{ R(j) ^{\sigma}}, \quad j=1, \dots, 2 p, $$
$$  \mbox{ch}_{L_{ \overline{M(1)} } (h_{i,1},0) ^+ }  - \mbox{ch}_{L_{ \overline{M(1)} } (h_{i,1},0) ^-} \  \quad (i=1, \dots, p).   $$
\end{theorem}
\begin{proof}

 We have showed above that all irreducible $\triplet ^{D_m}$ characters are in the sum of linear span of $\triplet ^{A_m}$ and $\triplet ^{A_2}$--characters. So if $m$ is even, we get assertion (1).
In both cases modular space is spanned by
 \bea && \frac{\Theta_{i,pm^2}(\tau)}{\eta(\tau)}  \quad (i =0, \dots, pm ^2) \label{theta-1} \\
 && \frac{\Theta_{i,4p}(\tau)}{\eta(\tau)} \quad (i =0, \dots,  4 p) \label{theta-2} \\
 && \frac{\partial \Theta_{i,p}(\tau)}{\eta(\tau)}, \quad (i = 1, \dots, p-1) \label{theta-3} \\
&& \frac{\tau \partial \Theta_{i,p}}{\eta(\tau)}, \quad (i=1, \dots, p-1 ). \label{theta-4} \eea

If $m$ is even, then characters (\ref{theta-2})  can be expressed by (\ref{theta-1}). So basis is given by (\ref{theta-1}), (\ref{theta-3}), (\ref{theta-4}).

If $m$ is odd, then formula
$$\Theta_{2i,4p}(\tau)+\Theta_{4p-2i,4p}(\tau)=\Theta_{i,p}(\tau)=\sum_{n=0}^{m-1}\Theta_{2pmn+im,pm^2}(\tau) $$
gives that  the dimension of the modular closure  is  $m^2p+5 p-1$. The proof follows.
\end{proof}

\end{document}